\documentclass{amsproc}
\newtheorem{theorem}{Theorem}[section]

\theoremstyle{definition}

\theoremstyle{remark}
\newtheorem{remark}[theorem]{Remark}

\numberwithin{equation}{section}

\usepackage{graphicx, pdflscape}
\usepackage{amsmath}
\usepackage{amssymb} 
\usepackage{relsize, float, lineno}
\usepackage{verbatim} 
\usepackage{xcolor} 
\usepackage{color}
\usepackage{subcaption}
\usepackage{lipsum}
\allowdisplaybreaks

\begin{document}
\title[NSFD methods preserving general quadratic Lyapunov functions]
{Nonstandard finite difference methods preserving\\ general quadratic Lyapunov functions}
\author{Manh Tuan Hoang}
\address{Department of Mathematics, FPT University, Hoa Lac Hi-Tech Park, Km29 Thang Long Blvd, Hanoi, Viet Nam}
\email{tuanhm14@fe.edu.vn; hmtuan01121990@gmail.com}
\subjclass[2020]{Primary 34A45, 65L05}

%
\keywords{Nonstandard finite difference, Lyapunov stability theory, Lyapunov functions, Global asymptotic stability, Dynamic consistency}



%

\begin{abstract}
In this work, we consider a class of dynamical systems described by ordinary differential equations under the assumption that the global asymptotic stability (GAS) of equilibrium points is  established based on the Lyapunov stability theory with the help of quadratic Lyapunov functions. We employ the Micken's methodology to construct a family of explicit nonstandard finite difference (NSFD) methods preserving any given quadratic Lyapunov function $V$, i.e. they admit $V$ as a discrete Lyapunov function. Here, the  proposed NSFD methods are derived from a novel non-local approximation for the zero vector function.\par
Through rigorous mathematical analysis, we show that the constructed NSFD methods have the ability to preserve any given quadratic Lyapunov functions regardless of the values of the step size. As an important consequence, they are dynamically consistent with respect to the GAS of continuous-time dynamical systems. On the other hand, the positivity of the proposed NSFD methods is investigated. It is proved that they can also preserve the positivity of solutions of continuous-time dynamical systems.\par
Finally, the theoretical findings are supported by a series of illustrative numerical experiments, in which advantages of the NSFD methods are demonstrated.
\end{abstract}
\maketitle
\section{Introduction}\label{sec1}
We start by considering a general finite-dimensional dynamical system described by ordinary differential equations (ODEs) associated with initial values
\begin{equation}\label{eq:1}
\dot{y}(t) = f(y(t)), \quad y(0) = y_0 \in \mathbb{R}^n,
\end{equation}
where $y = [y_1, y_2, \ldots, y_n]^T: [0, \infty) \to \mathbb{R}^n$, $f = [f_1, f_2, \ldots, f_n]^T: \mathbb{R}^n \to \mathbb{R}^n$ and $\dot{y}$ stands for the first derivative of $y$. Throughout this paper, we always assume that the right-hand side function $f$ satisfies all needed smoothness assumptions such that solutions of \eqref{eq:1} exist and are unique (see, e.g., \cite{Ascher, Khalil, Stuart}). It should be emphasized that the dynamical system \eqref{eq:1} is often utilized to mathematically model important phenomena and processes in science and engineering \cite{Ascher, Khalil, Stuart}.\par
We recall that a point $y^* \in \mathbb{R}^n$ is called an \textit{equilibrium point} of \eqref{eq:1} if $f(y^*) = 0$ and an equilibrium point is said to be \textit{globally asymptotically stable} if it is \textit{stable} and \textit{globally attractive} \cite{Khalil, Stuart}. Without the loss of generality, we can always assume that $y^* = 0$. It is well-known that the global asymptotic stability (GAS) analysis of equilibrium points is a very important problem with various useful applications in fields of theory and practice, for example in mathematical physics, economics, control theory, biology, ecology and so on (see, for instance, \cite{Khalil, LaSalle, LaSalle1, Lyapunov, Stuart}). The Lyapunov stability theory with the help of suitable Lyapunov functions has been recognized as one of the most successful approaches to this problem (see, for example, \cite{Cangiotti, Duarte-Mermoud, Korobeinikov, Korobeinikov1, Korobeinikov2, Korobeinikov3, ORegan, Vargas-De-Leon, Vargas-De-Leon1, Vargas-De-Leon2}). For the sake of convenience, below we recall the Lyapunov stability theorem, which is also known as Barbashin-Krasovskii theorem, for the GAS of continuous-time dynamical systems \cite[Theorem 4.2]{Khalil}.
\begin{theorem}\cite[Theorem 4.2]{Khalil}\label{Lyapunov1}
Let $y^* = 0$ be an equilibrium point for \eqref{eq:1}. Let $V: \mathbb{R}^n \to \mathbb{R}$ be a continuously differentiable function such that:
\begin{equation}\label{eq:2}
\begin{split}
&V(0) = 0,\\
&V(y) > 0, \quad \forall y \ne 0,\\
&\|y\| \to \infty \Rightarrow \|V(t)\| \to \infty,\\
&\dot{V(y)} < 0, \quad \forall y \ne 0, 
\end{split}
\end{equation}
then $y^* = 0$ is globally asymptotically stable.
\end{theorem}
Assume that the GAS of the dynamical system model \eqref{eq:1} has been established through a known Lyapunov function $V$ and we are using the following general one-step numerical method to approximate the solutions of \eqref{eq:1}
\begin{equation}\label{eq:3}
y^{k + 1} = \Phi(\Delta t, y^k),
\end{equation}
where $\Delta t > 0$ stands for the step size and $y^k$ is the intended approximation for $y(t_k)$ with $t_k = k\Delta t$ ($k \geq 1$). An essential requirement for the general numerical method \eqref{eq:3} is that it should preserve the Lyapunov function $V$ for all the step sizes $\Delta t$, i.e. \eqref{eq:3} admits the function $V$ as a discrete Lyapunov function. Consequently, we deduce from the Lyapunov stability theory for discrete-time dynamical systems \cite{Elaydi, Stuart} that the difference equation model \eqref{eq:3} preserves the GAS of \eqref{eq:1} regardless of the value of the step size. In this case, according to the Mickens' methodology \cite{Mickens1, Mickens2, Mickens3, Mickens4, Mickens5}, the numerical method is said to be dynamically consistent with respect to (w.r.t) the GAS of the counterpart differential equation model.\par
The study of numerical methods preserving Lyapunov functions is a very important problem, which has attracted the attention of many researchers (see, e.g., \cite{Calvo, Elliot, Grimm, McLachlan, McLachlan1, Stuart}. In an early work \cite{Grimm}, Grimm and Quispel proposed new projection-based methods that preserve a Lyapunov function as a Lyapunov function for discrete models generated by the numerical methods. Also, it was proved that this approach is very useful in various problems, in which a Lyapunov function is known. After that, Calvo et al. in \cite{Calvo} extended the previous work of Grimm and Quispel in \cite{Grimm} to construct new projected methods that preserve Lyapunov functions. These new methods are a modification of the Grimm and Quispel ones are derived from an explicit Runge-Kutta method provided with dense output, together with a quadrature formula with positive weights; however, the non-negativity of the coefficients $b_i$ of the underlying RK method is relaxed. Some numerical methods of low order for which a given Lyapunov function of a dynamical system is also a discrete Lyapunov function have been widely studied in \cite{Elliot, McLachlan, McLachlan1, Stuart}.\par
In this present work, we introduce a simple approach, which is based on the Mickens' methodology \cite{Mickens1, Mickens2, Mickens3, Mickens4, Mickens5} and differs from the approaches of geometric integration utilized in the above-mentioned works, to construct to a family of explicit nonstandard finite difference (NSFD) methods, which preserves a general quadratic Lyapunov function of the form:
\begin{equation}\label{eq:4}
V_{QLF}(y) = V_{QLF}(y_1, y_2, \ldots, y_n) := \sum_{i = 1}^n\alpha_i (y_i - y_i^*)^2,
\end{equation}
where $\alpha_i > 0$ are coefficients of the Lyapunov function, $y^* = (y_1^*, y_2^*, \ldots, y_n^*)$ is a globally asymptotically stable equilibrium point of \eqref{eq:1}. The NSFD methods are derived from a novel weighted non-local approximation for the zero vector function. Through rigorous mathematical analysis, we determine easily-verified conditions which ensure that the NSFD methods preserve the quadratic Lyapunov function $V_{QLF}$ defined in \eqref{eq:4} for all the finite step sizes $\Delta t$, that is the variation of $V_{QLF}$ satisfies
\begin{equation}\label{eq:5}
\Delta V_{QLF}(y^k) :=  V_{QLF}(y^{k + 1}) - V_{QLF}(y^k) < 0
\end{equation}
for all $y^k := (y_1^k, y_2^k, \ldots, y_n^k) \ne y^*$ and $\Delta t > 0$. Consequently, the NSFD methods preserve the GAS of the counterpart differential equation models. It should be emphasized that the approach based on weighted non-local approximations for the zero function has been utilized in our recent works to formulate NSFD methods for ODEs \cite{Hoang5, Hoang6, HoangMatthias, HoangMatthias1}. \par
It is worthy noting that the construction of NSFD methods preserving the GAS of differential equation has been intensively investigated in some previous works (see, for example, \cite{Biwas, Cui, DangHoang1, DangHoang2, DangHoang3, DangHoang4, Ding, Hoang5new, Hoang5, Yang}). However, to the best of our knowledge, results regarding NSFD methods preserving general Lyapunov functions are very scarce up to now. We can consider the last condition of \eqref{eq:2} as a monotonic property w.r.t the function $V$; therefore, the NSFD methods also preserve this property.\par
Along with the GAS with Lyapunov functions, we also consider dynamical systems of the form \eqref{eq:1} whose solutions satisfy the positivity \cite{Horvath, Smith}. We show proved that the constructed NSFD methods can also preserve the positivity of solutions of continuous-time dynamical systems.\par
The plan of this work is as follows:  The NSFD methods are proposed and analyzed in Section \ref{sec2}. A series of numerical examples are conducted in Section \ref{sec3}. Section \ref{sec4} is devoted to investigating the positivity of the NSFD models. In the last section, we provide some concluding remarks and open problems.
\section{NSFD methods preserving quadratic Lyapunov functions}\label{sec2}
This section is devoted to constructing a family of explicit NSFD methods that preserve the general quadratic Lyapunov function $V_{QLF}$ defined in \eqref{eq:4}. Note that the derivative of $V_{QLF}$ along the solutions of \eqref{eq:1} satisfies
\begin{equation}\label{eq:6}
\dot{V}_{QLF}(y) =  2\sum_{i = 1}^n\alpha_i(y_i - y_i^*)f_i(y) < 0, \quad \forall y \ne y^*.
\end{equation}
So, numerical methods should provide numerical approximations that satisfy \eqref{eq:5}.\par
First, we replace the first derivatives $\dot{y}_i(t)$ ($1 \leq i \leq n$) by
\begin{equation}\label{eq:7}
\dot{y}_i(t)|_{t = t_k} \approx \dfrac{y_i^{k + 1} - y_i^k}{\phi(\Delta t)},
\end{equation}
where $\phi(\Delta t)$ is called a denominator function, which is positive and $\phi(\Delta t) = \Delta t + \mathcal{O}(\Delta t^2)$ as $\Delta t \to 0$ \cite{Mickens1, Mickens2, Mickens3, Mickens4, Mickens5}. Next, by applying the rules of construction of NSFD \cite{Mickens1, Mickens2, Mickens3, Mickens4, Mickens5}, we approximate the right-hand side function $f_i(y(t))$ $(1 \leq i \leq n)$ by
\begin{equation}\label{eq:8}
\begin{split}
f_i(y(t_k)) &= f_i(y(t_k)) + \big(y_i(t_k)\tau(y(t_k)) - y_i(t_k)\tau(y(t_k))\big)\\
&\approx f_i(y^k) + \big(y_i^k\tau(y^k) - y_i^{k + 1}\tau(y^k)\big),
\end{split}
\end{equation}
where $\tau(y): \mathbb{R}^n \to \mathbb{R}$ is a function satisfying $\tau(y) \geq 0$ for all $y \in \mathbb{R}^n$, which plays a role as a weight. The term $y_i^k\tau(y^k) - y_i^{k + 1}\tau(y^k)$ can be considered as a non-local approximation of the zero function. The approximations \eqref{eq:7} and \eqref{eq:8} lead to the following family of NSFD methods
\begin{equation}\label{eq:9}
\dfrac{y_i^{k + 1} - y_i^k}{\phi(\Delta t)} = f_i(y^k) + y_i^k\tau(y^k) - y_i^{k + 1}\tau(y^k).
\end{equation}
The explicit form of  \eqref{eq:9} is given by
\begin{equation}\label{eq:10}
y_i^{k + 1} = \dfrac{y_i^k + \phi(\Delta t)\big(f_i(y^k) + y_i^k\tau(y^k)\big)}{1 + \phi(\Delta t)\tau(y^k)} = y_i^k + \dfrac{\phi(\Delta t)f_i(y^k)}{1 + \phi(\Delta t) \tau(y^k)}.
\end{equation}
We now determine conditions for the function $\tau(y)$ such that \eqref{eq:9} or also \eqref{eq:10} preserves the general quadratic Lyapunov function \eqref{eq:4}.
\begin{theorem}\label{MainTheorem1}
Let $\tau(y)$ be a function with the property that for $y \ne y^*$
\begin{equation}\label{eq:10new}
\tau(y) \geq \tau_L(y) := -\dfrac{\mathlarger{\mathlarger{\sum}}_{i = 1}^n\alpha_i (f_i(y))^2}{\dot{V}_{QLF}(y)},
\end{equation}
where $\dot{V}_{QLF}(y)$ is given in \eqref{eq:6}. Then, the NSFD model \eqref{eq:9} admits the quadratic Lyapunov function $V_{QLF}$ defined in \eqref{eq:4} as a discrete Lyapunov function for all the values of the step size $\Delta t$.
\end{theorem}
\begin{proof}
We only need to verify that \eqref{eq:5} is satisfied. Indeed, by using \eqref{eq:10} we have
\begin{equation}\label{eq:12}
\begin{split}
\Delta V_{QLF}(y^k) &:= V_{QLF}(y^{k + 1}) - V_{QLF}(y^k)\\
&= \sum_{i = 1}^n\alpha_i(y_i^{k + 1} - y_i^*)^2 - \sum_{i = 1}^n\alpha_i\big(y_i^k - y_i^*)^2\\
&= \sum_{i = 1}^n\alpha_i\big[(y_i^{k + 1} - y_i^*)^2 - \big(y_i^k - y_i^*)^2\big]\\
&= \sum_{i = 1}^n\alpha_i\big[(y_i^{k + 1} - y_i^k)(y_i^{k + 1} + y_i^k - 2y_i^*)\big]\\
&= \mathlarger{\mathlarger{\sum}}_{i = 1}^n\alpha_i\dfrac{\phi(\Delta t)f_i(y^k)}{1 + \phi(\Delta t) \tau(y^k)}\bigg(2y_i^k - 2y_i^* + \dfrac{\phi(\Delta t)f_i(y^k)}{1 + \phi(\Delta t)\tau(y^k)}\bigg)\\
&= \mathlarger{\mathlarger{\sum}}_{i = 1}^n\alpha_i\dfrac{\phi(\Delta t)f_i(y^k)}{1 + \phi(\Delta t)\tau(y^k)}\big(2y_i^k - 2y_i^*) + \mathlarger{\mathlarger{\sum}}_{i = 1}^n\alpha_i\bigg(\dfrac{\phi(\Delta t)f_i(y^k)}{1 + \phi(\Delta t) \tau(y^k)}\bigg)^2\\
&= \dfrac{\phi(\Delta t)}{1 + \phi(\Delta t)\tau(y^k)}\mathlarger{\mathlarger{\sum}}_{i = 1}^n2\alpha_if_i(y^k)(y_i^k - y_i^*) + \mathlarger{\mathlarger{\sum}}_{i = 1}^n\alpha_i\bigg(\dfrac{\phi(\Delta t)f_i(y^k)}{1 + \phi(\Delta t) \tau(y^k)}\bigg)^2\\
&= \dfrac{\phi(\Delta t)}{1 + \phi(\Delta t)\tau(y^k)}\dot{V}_{QLF}(y^k) + \mathlarger{\mathlarger{\sum}}_{i = 1}^n\alpha_i\bigg(\dfrac{\phi(\Delta t)f_i(y^k)}{1 + \phi(\Delta t) \tau(y^k)}\bigg)^2\\
&= \dfrac{\phi(\Delta t)}{1 + \phi(\Delta t)\tau(y^k)}\bigg[\dot{V}_{QLF}(y^k) + \mathlarger{\mathlarger{\sum}}_{i = 1}^n\alpha_i\dfrac{\phi(\Delta t)}{1 + \phi(\Delta t)\tau(y^k)}\big(f_i(y^k)\big)^2\bigg].
\end{split}
\end{equation}
It follows from \eqref{eq:12} that $\Delta V_{QLF}(y^k) < 0$ if and only if
\begin{equation}\label{eq:13}
\mathcal{L} := \dot{V}_{QLF}(y^k) + \dfrac{\phi(\Delta t)}{1 + \phi(\Delta t)\tau(y^k)}\mathlarger{\mathlarger{\sum}}_{i = 1}^n\alpha_i\big(f_i(y^k)\big)^2 < 0,
\end{equation}
which is equivalent to (since $1 + \phi(\Delta t)\tau(y^k) > 0$)
\begin{equation*}\label{eq:14}
\dot{V}_{QLF}(y^k) + \phi(\Delta t)\bigg[\tau(y^k)\dot{V}_{QLF}(y^k) + \mathlarger{\mathlarger{\sum}}_{i = 1}^n\alpha_i\big(f_i(y^k)\big)^2\bigg] < 0.
\end{equation*}
It is easy to check that if \eqref{eq:10new} then $\dot{V}_{QLF}(y^k)\tau(y^k) + \sum_{i = 1}^n\alpha_i(f_i(y^k))^2 \leq 0$. Consequently, $\mathcal{L}$ in \eqref{eq:13} is negative. Therefore, we conclude that $\Delta V_{QLF}(y^k)$ in \eqref{eq:12} is negative provided that \eqref{eq:10new} holds. The proof is complete.
\end{proof}
\begin{remark}
The function $\tau_L(y)$ in \eqref{eq:10new} satisfies $\tau_L(y) \geq 0$ for all $y \in \mathbb{R}$. So, $\tau(y) \geq 0$ for all $y \in \mathbb{R}$.
%
For each function $\tau(y)$, we obtain an NSFD method preserving the quadratic Lyapunov functions. In general, we can take
\begin{equation*}
\tau(y) = \tau_L(y) + g(y),
\end{equation*}
where $g(y)$ satisfies $g(y) \geq 0$ for $y \in \mathbb{R}^n$.
\end{remark}
\begin{remark}\label{remark2new}
The non-local approximation \eqref{eq:8} and the positivity of $\tau(y)$ ensure that the NSFD model \eqref{eq:9} is defined and consistent; hence, it is convergent of order one (see, for e.g., \cite{Ascher, Cresson}).
\end{remark}
\section{Numerical examples}\label{sec3}
In this section, we report illustrate numerical examples to support the findings presented in Section \ref{sec2}. For this purpose, consider the following system, which generates a system considered in \cite[Example 2]{Ghaffari}, as a test problem
\begin{equation}\label{eq:32}
\begin{split}
&\dot{y}_1 = -Ay_1^3 + By_2,\\
&\dot{y}_2 = -Cy_1 - Dy_2^3,
\end{split}
\end{equation}
where $A, B, C$ and $D$ are positive constants. The system \eqref{eq:32} admits the quadratic function
\begin{equation}\label{eq:32new}
V(y_1, y_2) = \alpha_1 y_1^2 + \alpha_2 y_2^2, \quad \alpha_1 = \dfrac{C}{B}\alpha_2,
\end{equation}
as a Lyapunov function, which implies the GAS of the origin $(0, 0)$.\par
We will apply the NSFD model \eqref{eq:9} for the system \eqref{eq:32}. In this case, we have
\begin{equation*}
\begin{split}
&f_1(y_1, y_2) = -Ay_1^3 + By_2,\\
&f_2(y_1, y_2) = -Cy_1 - Dy_2^3,\\
&\dot{V}(y_1, y_2) = -2\alpha_1A y_1^4 - 2\alpha_2D y_2^4. 
\end{split}
\end{equation*}
So, the function $\tau_L$ in \eqref{eq:10new} is given by
\begin{equation*}
\tau_L(y) = \dfrac{\alpha_1(-Ay_1^3 + By_2)^2 + \alpha_2(-Cy_1^4 - Dy_2^4)^2}{2A\alpha_1y_1^4 + 2D\alpha_2y_2^4}.
\end{equation*}
We choose the function $\tau(y)$ for \eqref{eq:9} as
\begin{equation*}
\tau(y) = \tau_L(y) + 0.001,
\end{equation*}
which satisfies \eqref{eq:10new}.\par
We now consider \eqref{eq:32} with the following set of the parameters
\begin{equation*}
(A, B, C, D) = (0.16,\, 1,\, 1,\, 0.1)
\end{equation*}
and the initial data
\begin{equation*}
(y_1(0), y_2(0)) = (0.5,\,\, 0.01)
\end{equation*}
First, let us observe the behaviour of \eqref{eq:32} over a long time period, namely $0 \leq t \leq 1000$. For this, we employ the classical four-stage Runge-Kutta (RK4) method \cite{Ascher, Stuart} with $\Delta t = 10^{-4}$ to obtain a reference solution. The reference solution is sketched in Figure \ref{Fig:1}. It is clear that the components $y_1(t)$ and $y_2(t)$ are stable and decrease to $0$, whereas the phase plane forms a spiral into the origin.\par
Next, we utilize the standard Euler method and  a two-stage Runge-Kutta method (RK2), namely the standard explicit trapezoidal method, to solve the system \eqref{eq:32}. Numerical approximations are depicted in Figures \ref{Fig:2}-\ref{Fig:3.1}. It is observed that these standard methods provide the numerical approximations that form two spirals away from the initial position (see Figures \ref{Fig:2} and \ref{Fig:3}). Also, the quadratic Lyapunov function in \eqref{eq:32new} is not preserved (see Figures \ref{Fig:2.1} and \ref{Fig:3.1}).\par
Now, for the purpose of comparison, we apply the NSFD method \eqref{eq:9} to simulate the system \eqref{eq:32}, which uses $\phi(\Delta t) = \Delta t$ and three different step sizes, namely $\Delta t \in \{1.0,\,\,0.8,\,\,0.001\}$. Numerical solutions are sketched in Figures \ref{Fig:4}-\ref{Fig:7}. We see from these figures that the GAS of the origin and the dynamics of the continuous-time system are exactly preserved (see Figures \ref{Fig:4}). On the other hand, the Lyapunov function in \eqref{eq:32new} is also preserved (see Figures \ref{Fig:5}-\ref{Fig:7}).\par
Before ending this section, we simulate the dynamics of \eqref{eq:32} by employing the NSFD model \eqref{eq:9} with $\Delta t \in \{1.0,\,\, 0.001\}$. The numerical approximations are shown in Figure \ref{Fig:89}. Clearly, the GAS of the origin and the behaviour of the continuous-time model are reflected exactly regardless of the chosen step sizes.\par
In conclusion, the numerical examples are consistent and hence, provide evidence supporting  the theoretical findings.
%
\begin{figure}[H]
\subfloat[$y_1$ component.]{%
\includegraphics[height=9cm,width=6cm]{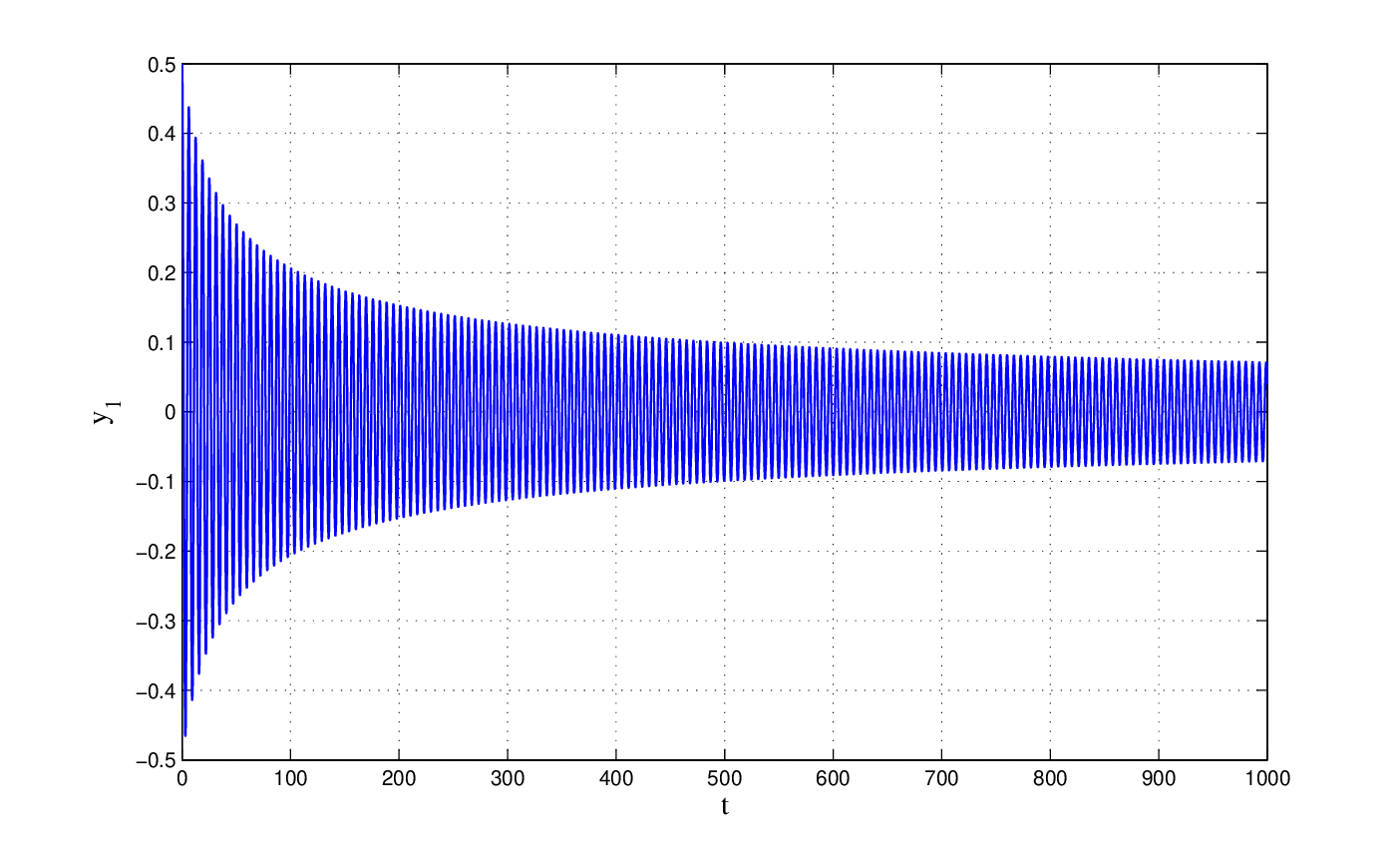}
\label{Figure:1a}
}\hfill
\subfloat[$y_2$ component.]{%
\includegraphics[height=9cm,width=6cm]{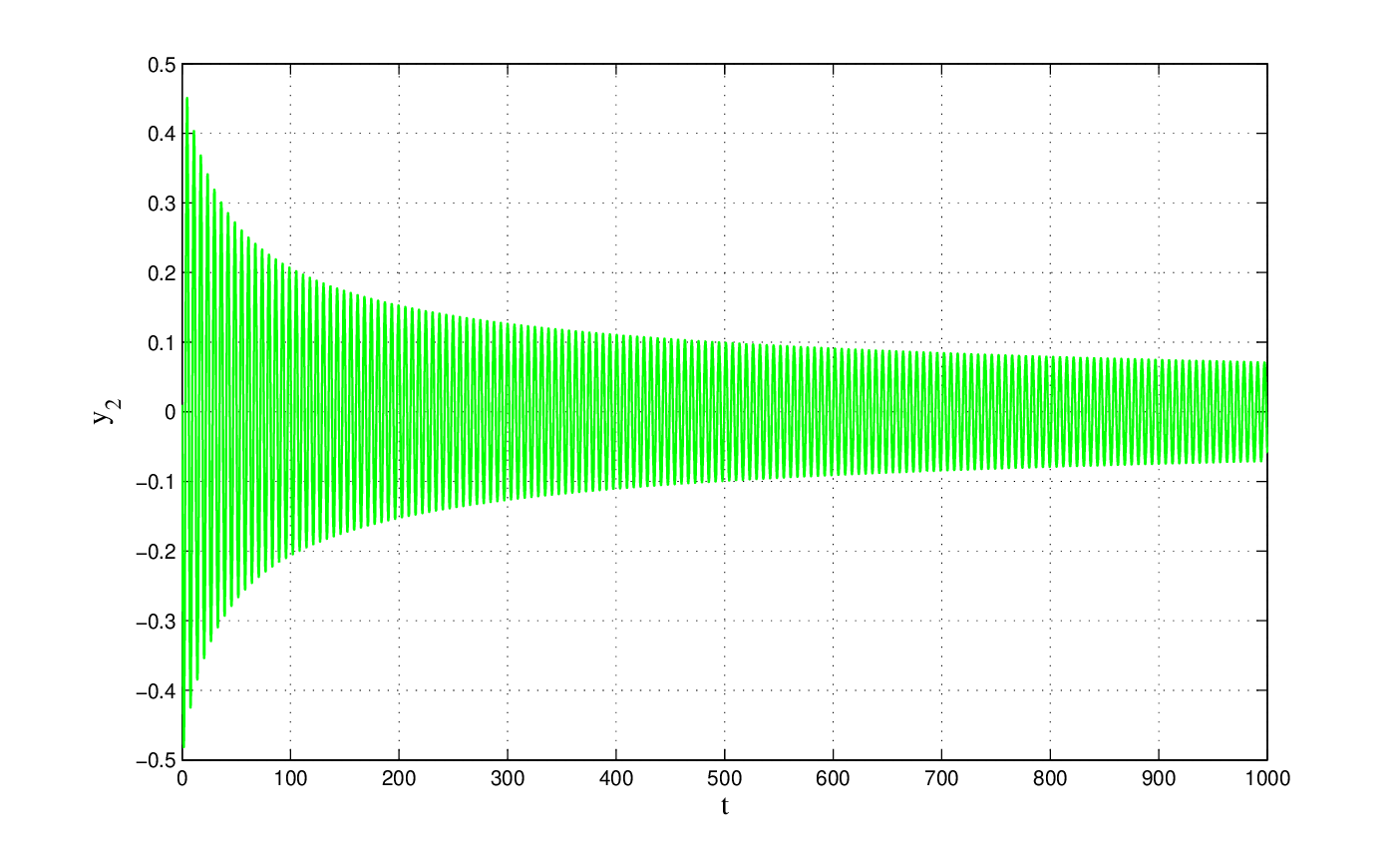}
\label{Figure:1b}
}\hfill
\subfloat[The phase plane.]{%
\includegraphics[height=9cm,width=12cm]{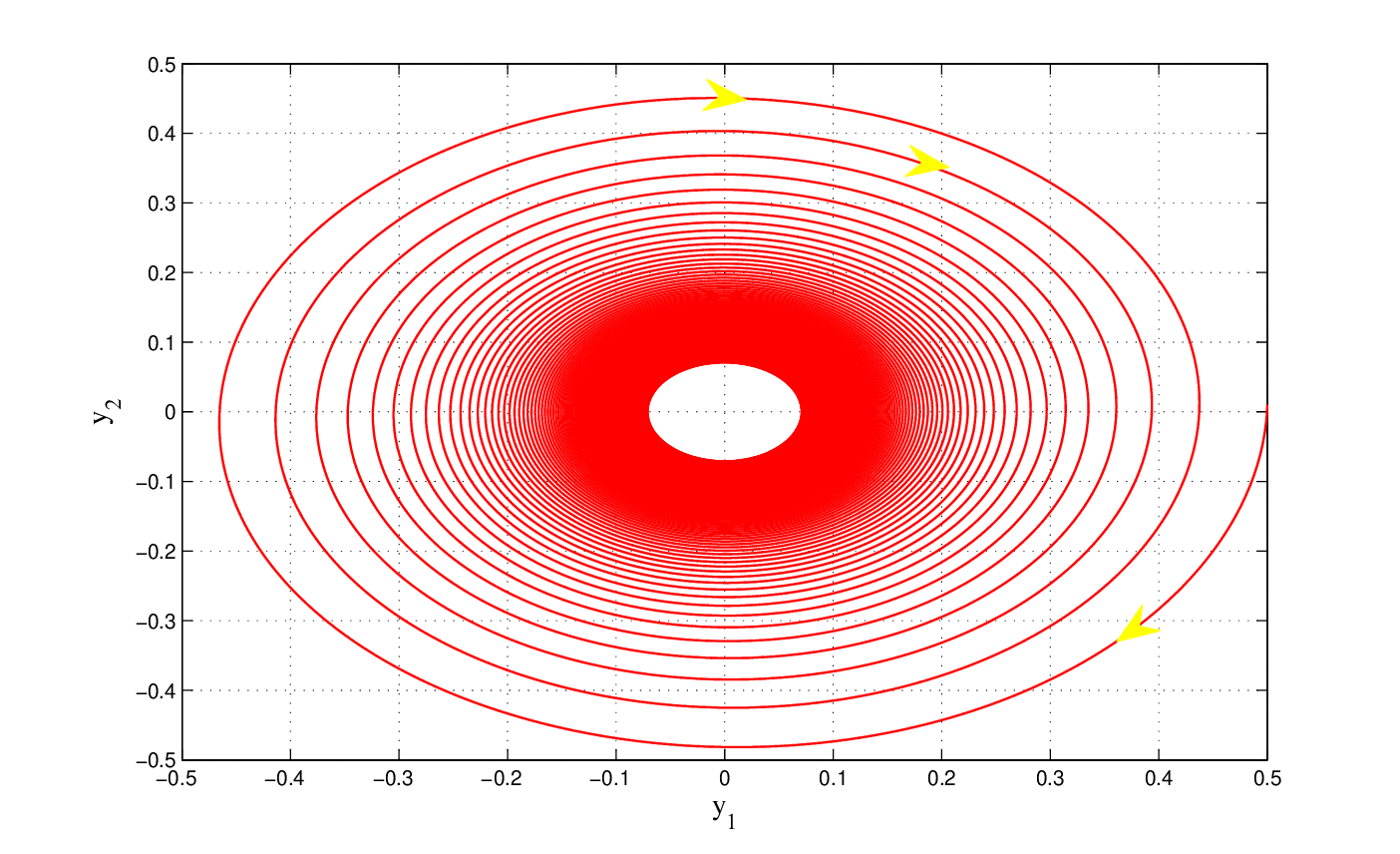}
\label{Figure:1c}
}
\caption{The reference solution generated by employing the RK4 method with $\Delta t = 10^{-4}$.}
\label{Fig:1}
\end{figure}
%
%
\begin{figure}[H]
\subfloat[$y_1$ component.]{%
\includegraphics[height=9cm,width=6cm]{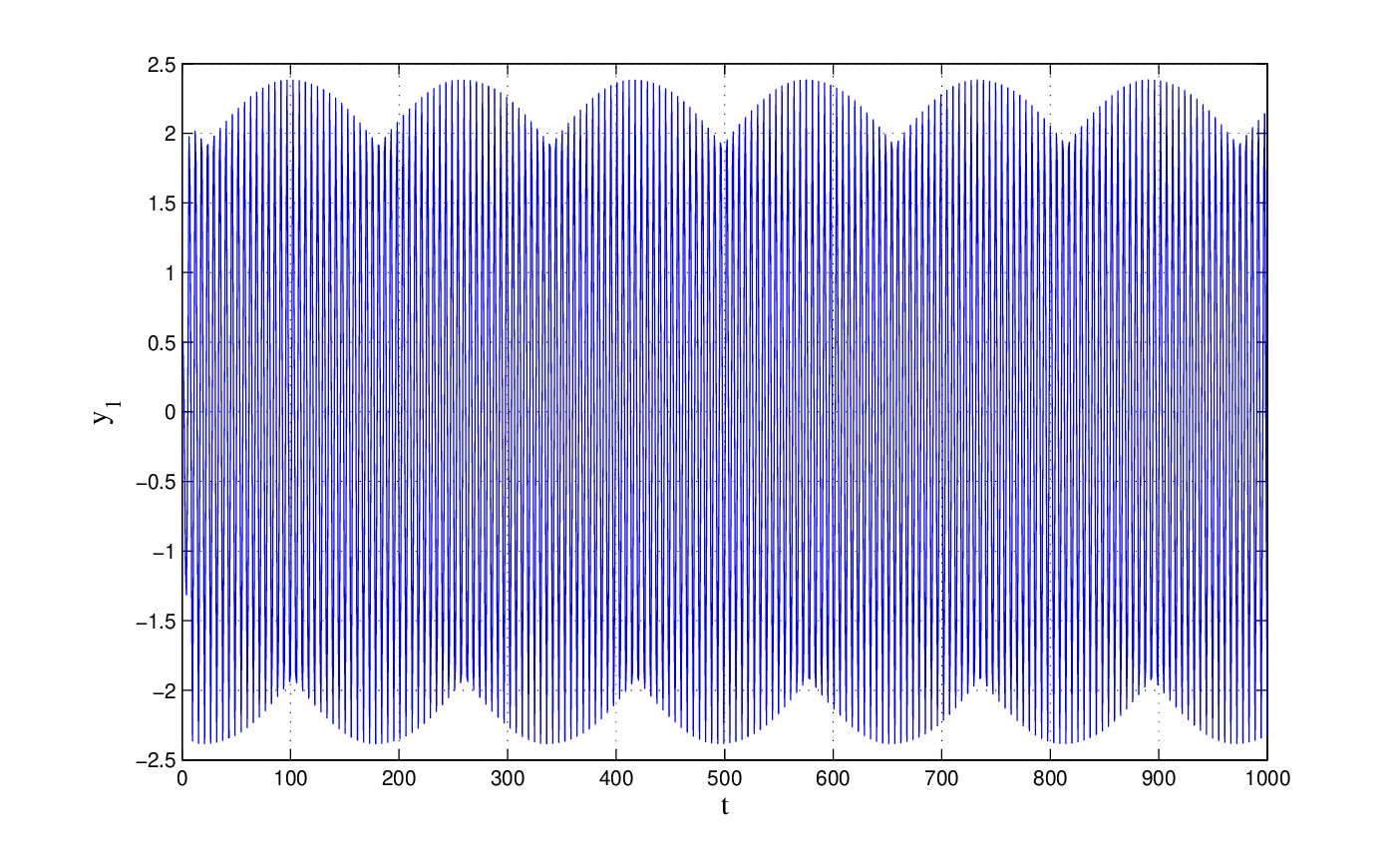}
\label{Figure:2a}
}\hfill
\subfloat[$y_2$ component.]{%
\includegraphics[height=9cm,width=6cm]{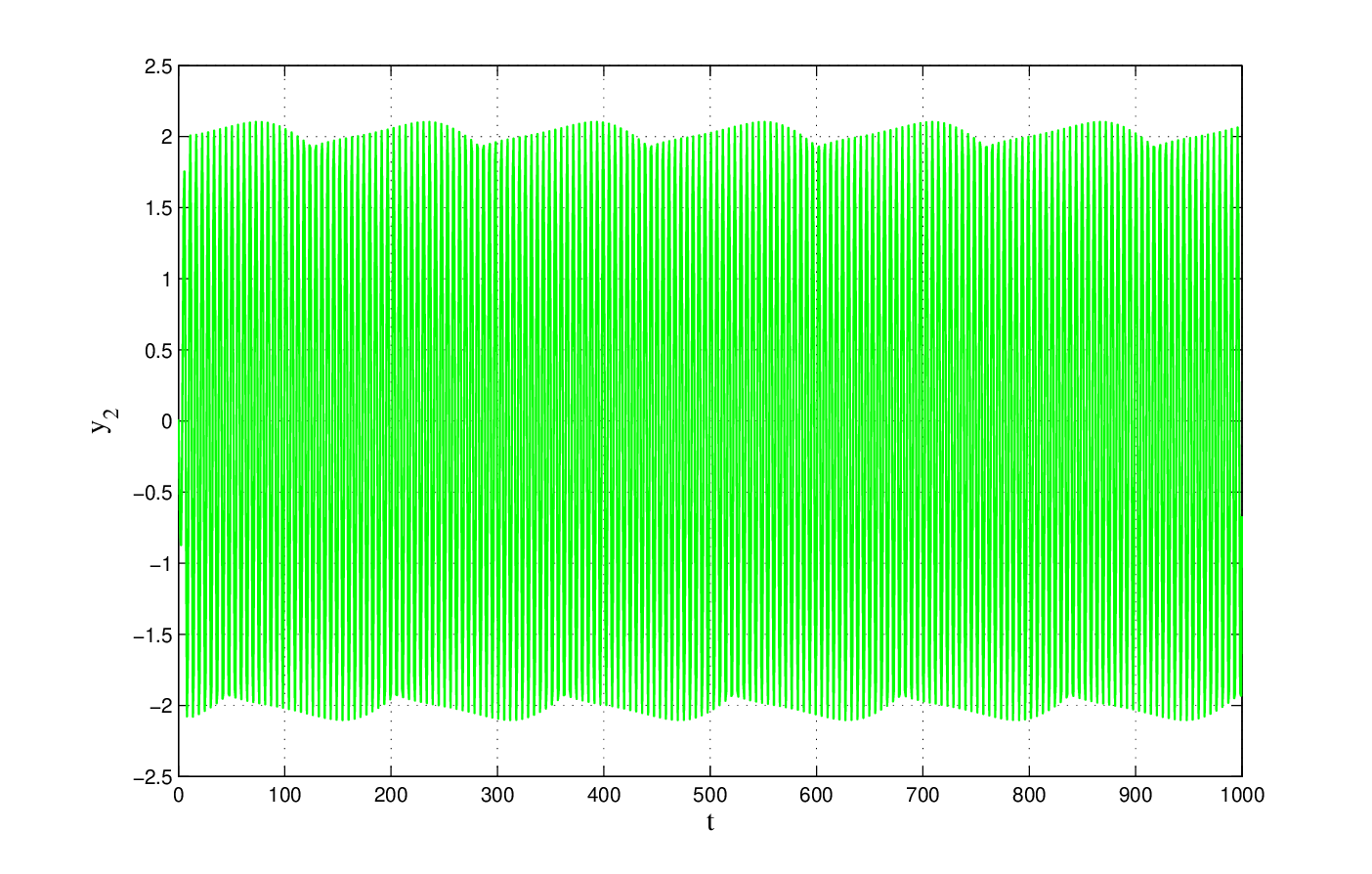}
\label{Figure:2b}
}\hfill
\subfloat[The phase plane.]{%
\includegraphics[height=9cm,width=12cm]{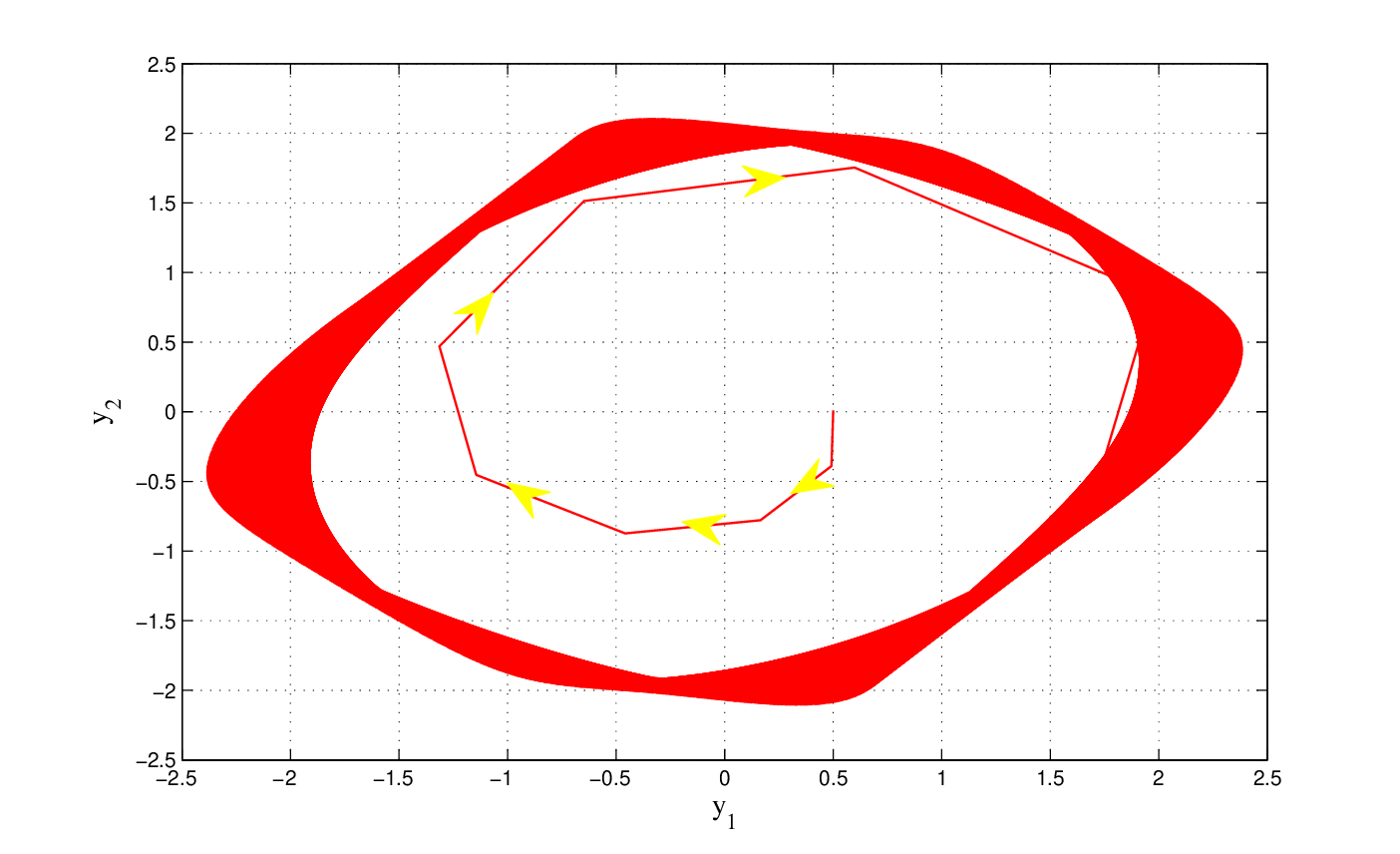}
\label{Figure:2c}
}
\caption{The approximations generated by employing the Euler method with $\Delta t = 0.8$.}
\label{Fig:2}
\end{figure}
\begin{figure}[H]
\subfloat[$V(y^k) = (y_1^k)^2 + (y_2^k)^2$.]{%
\includegraphics[height=9cm,width=12cm]{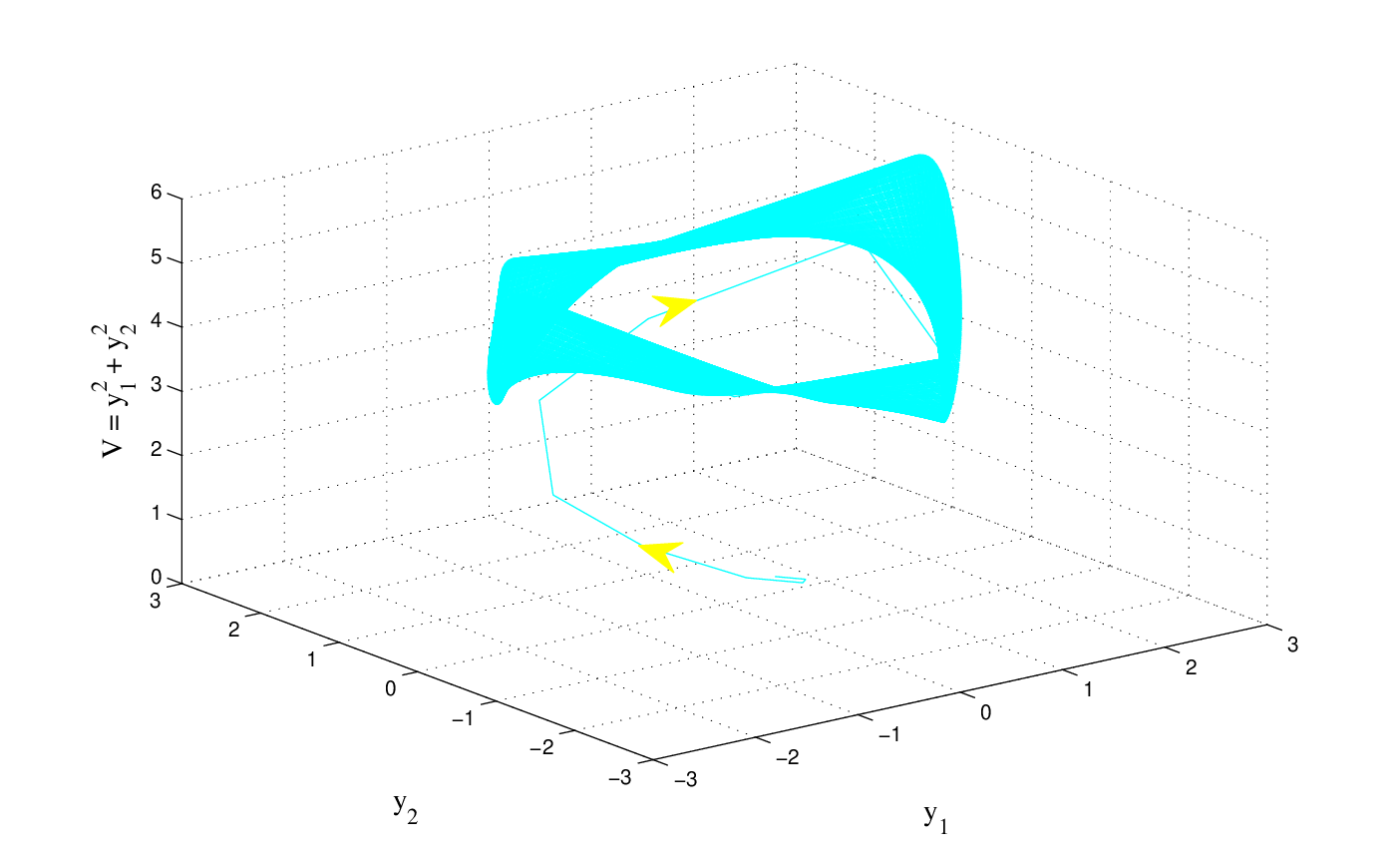}
\label{Figure:2d}
}\hfill
\subfloat[$\Delta V(y^k) = V(y^{k + 1}) - V(y^k)$.]{%
\includegraphics[height=9cm,width=12cm]{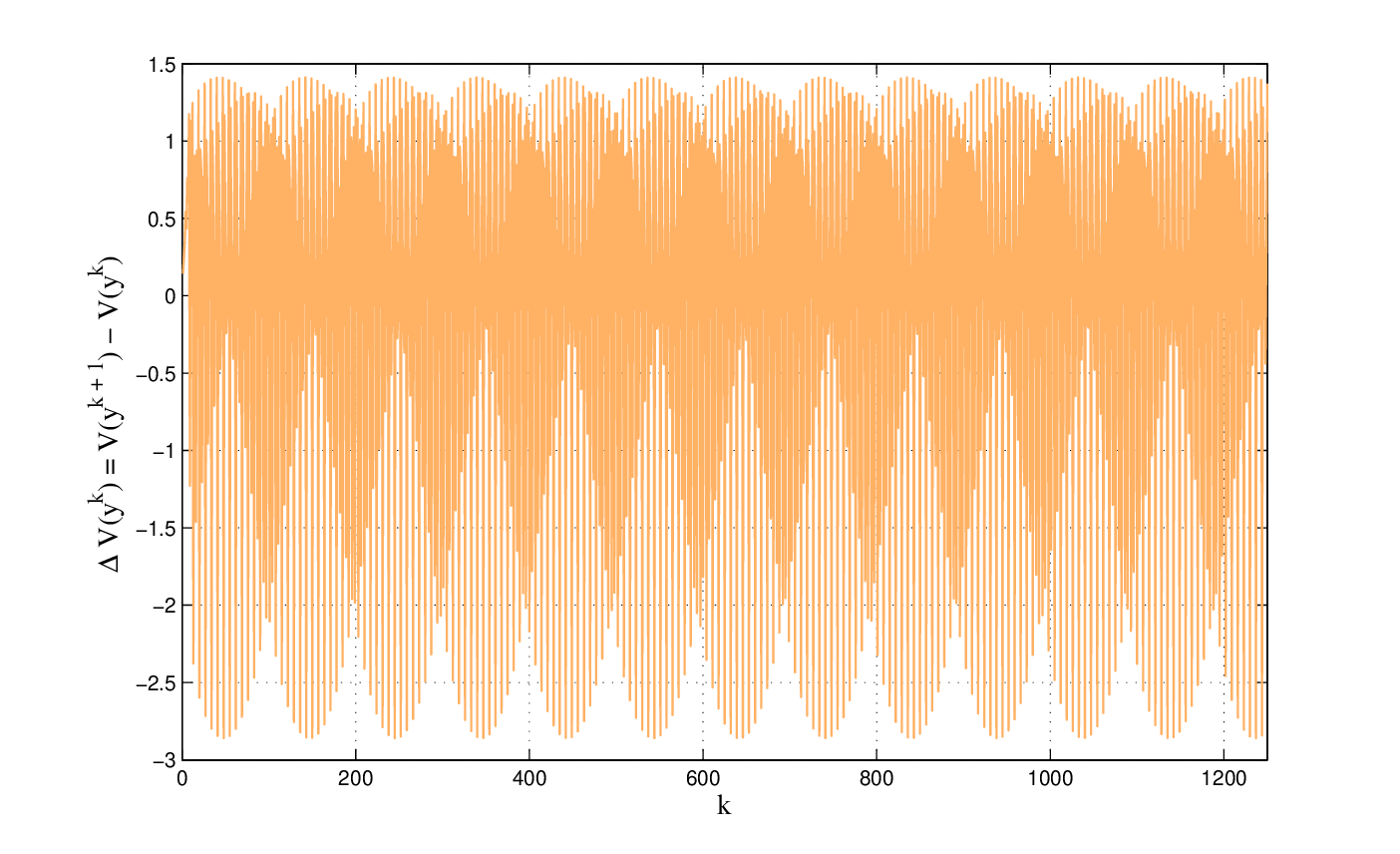}
\label{Figure:2e}
}
\caption{The discrete Lyapunov function and its variation corresponding to the Euler method using $\Delta t = 0.8$.}
\label{Fig:2.1}
\end{figure}
\begin{figure}[H]
\subfloat[$y_1$ component.]{%
\includegraphics[height=9cm,width=6cm]{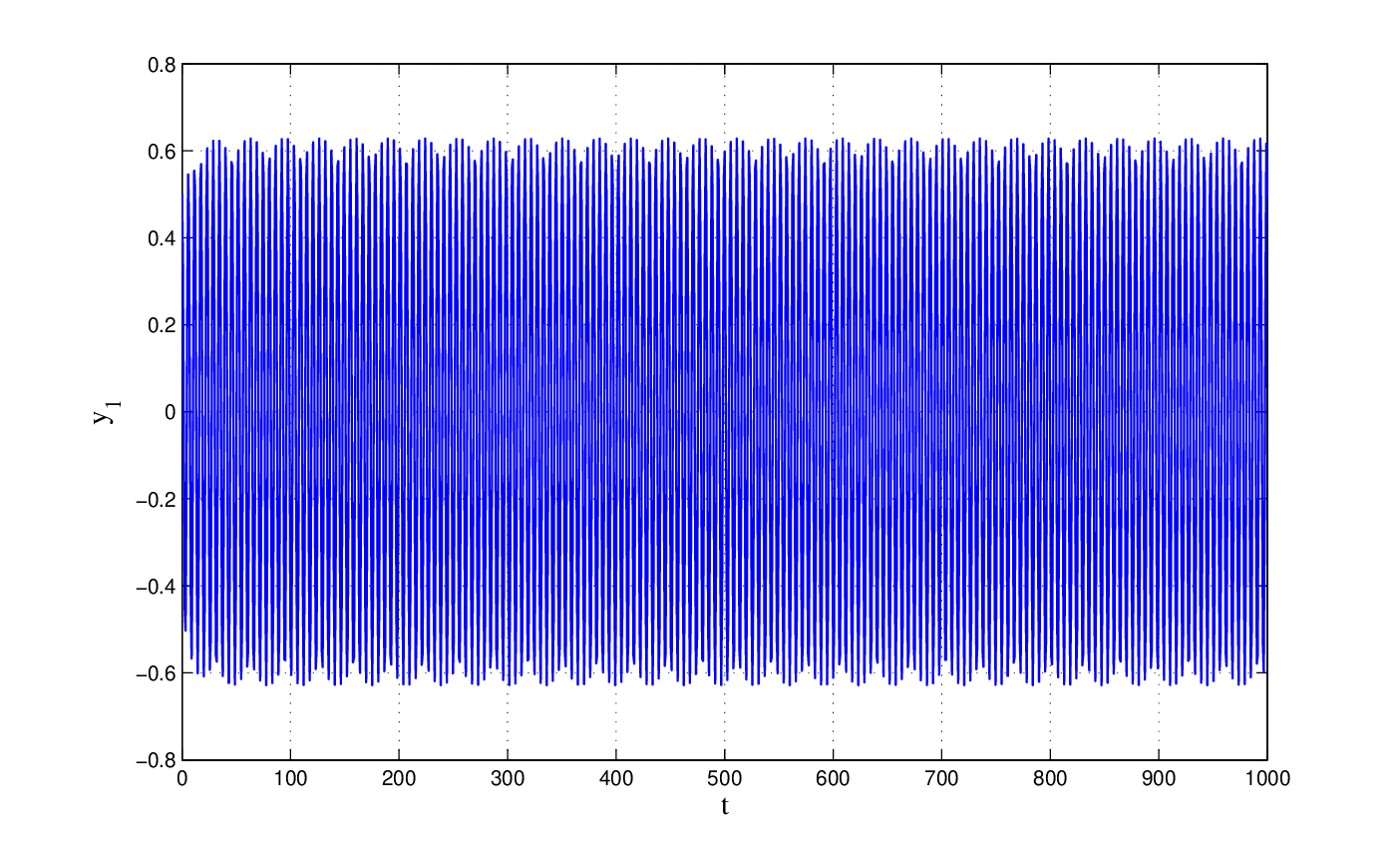}
\label{Figure:3a}
}\hfill
\subfloat[$y_2$ component.]{%
\includegraphics[height=9cm,width=6cm]{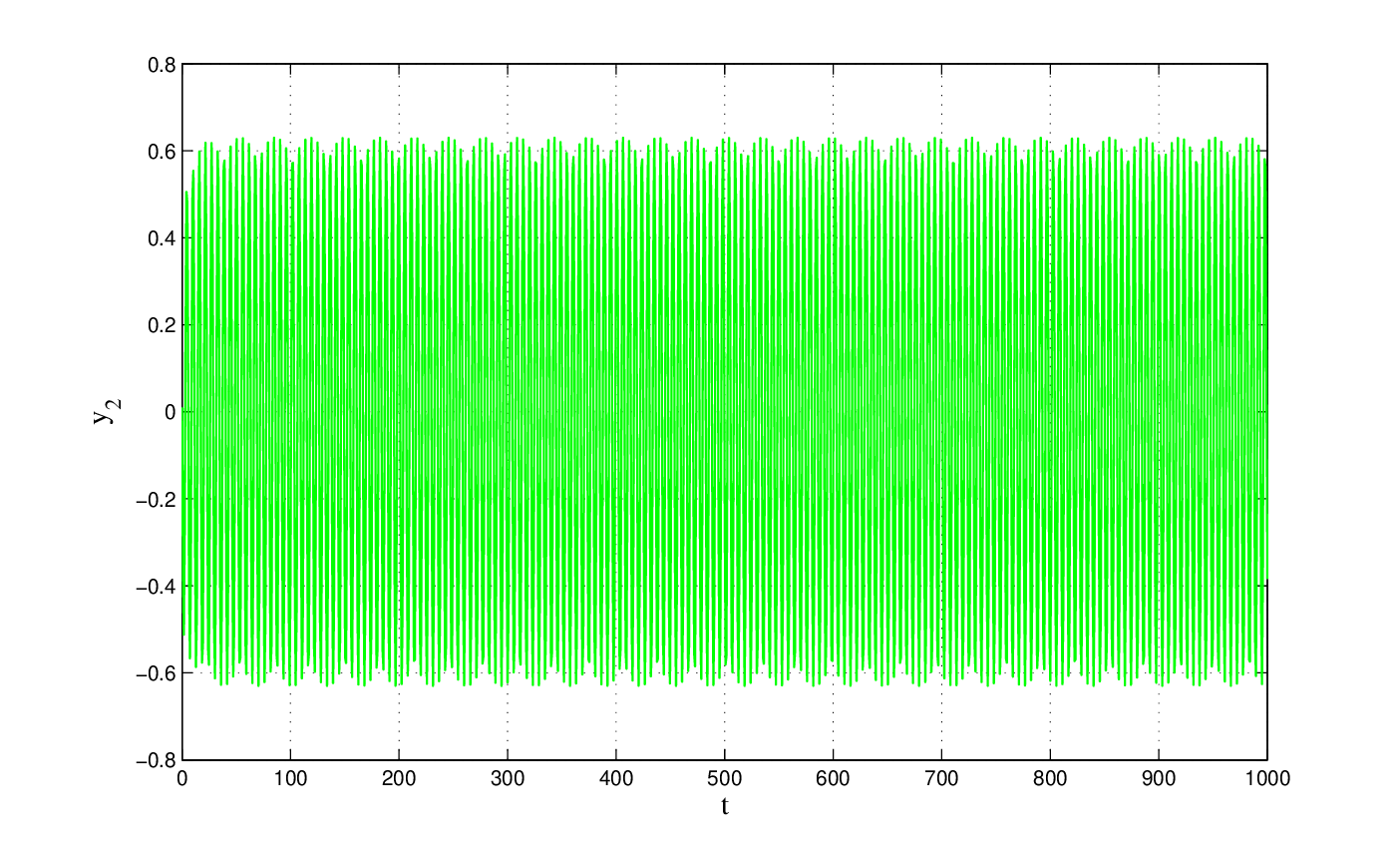}
\label{Figure:3b}
}\hfill
\subfloat[The phase plane.]{%
\includegraphics[height=9cm,width=12cm]{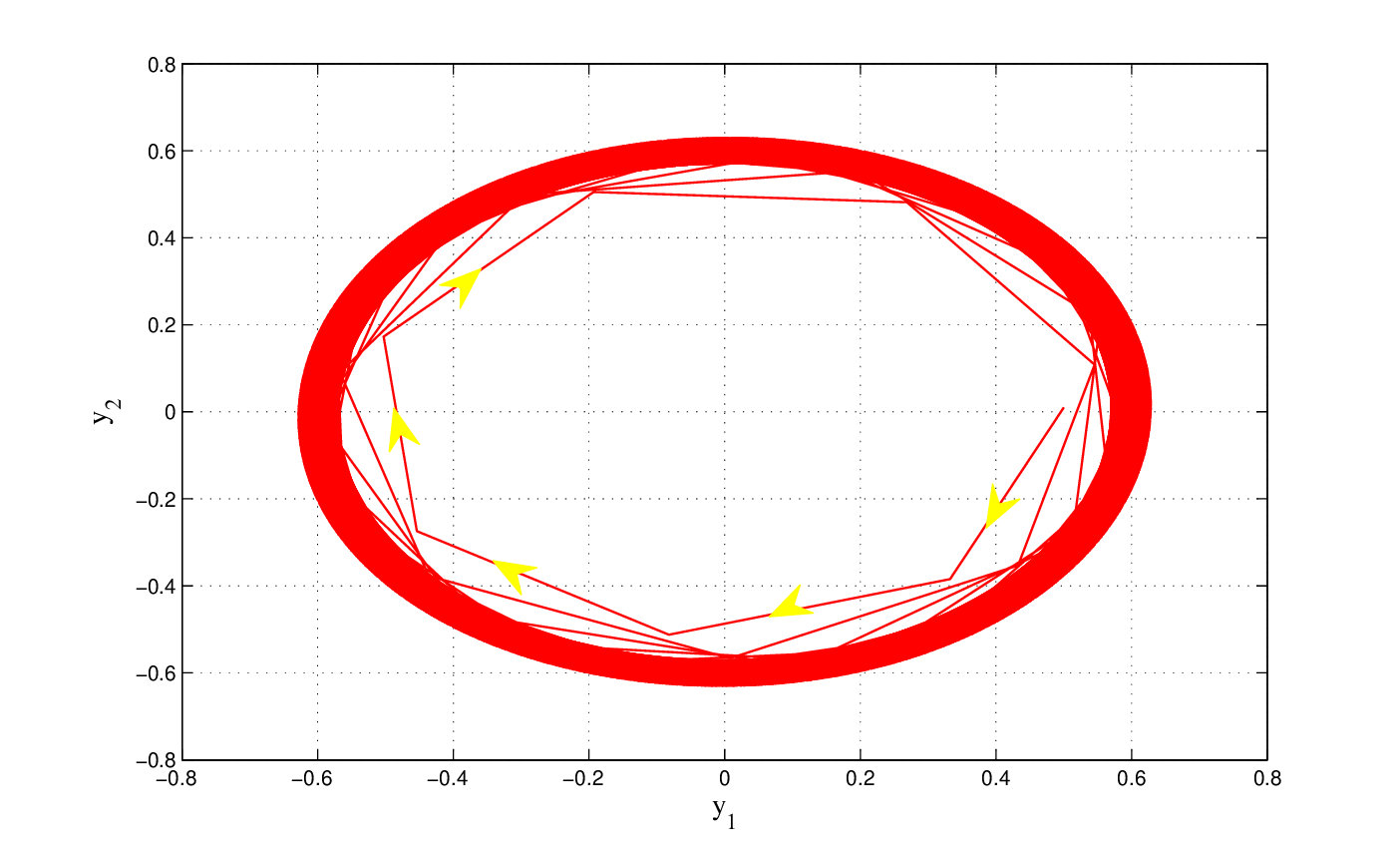}
\label{Figure:3c}
}
\caption{The approximated generated by employing the RK2 method with $\Delta t = 0.8$.}
\label{Fig:3}
\end{figure}
\begin{figure}[H]
\subfloat[$V(y^k) = (y_1^k)^2 + (y_2^k)^2$.]{%
\includegraphics[height=9cm,width=12cm]{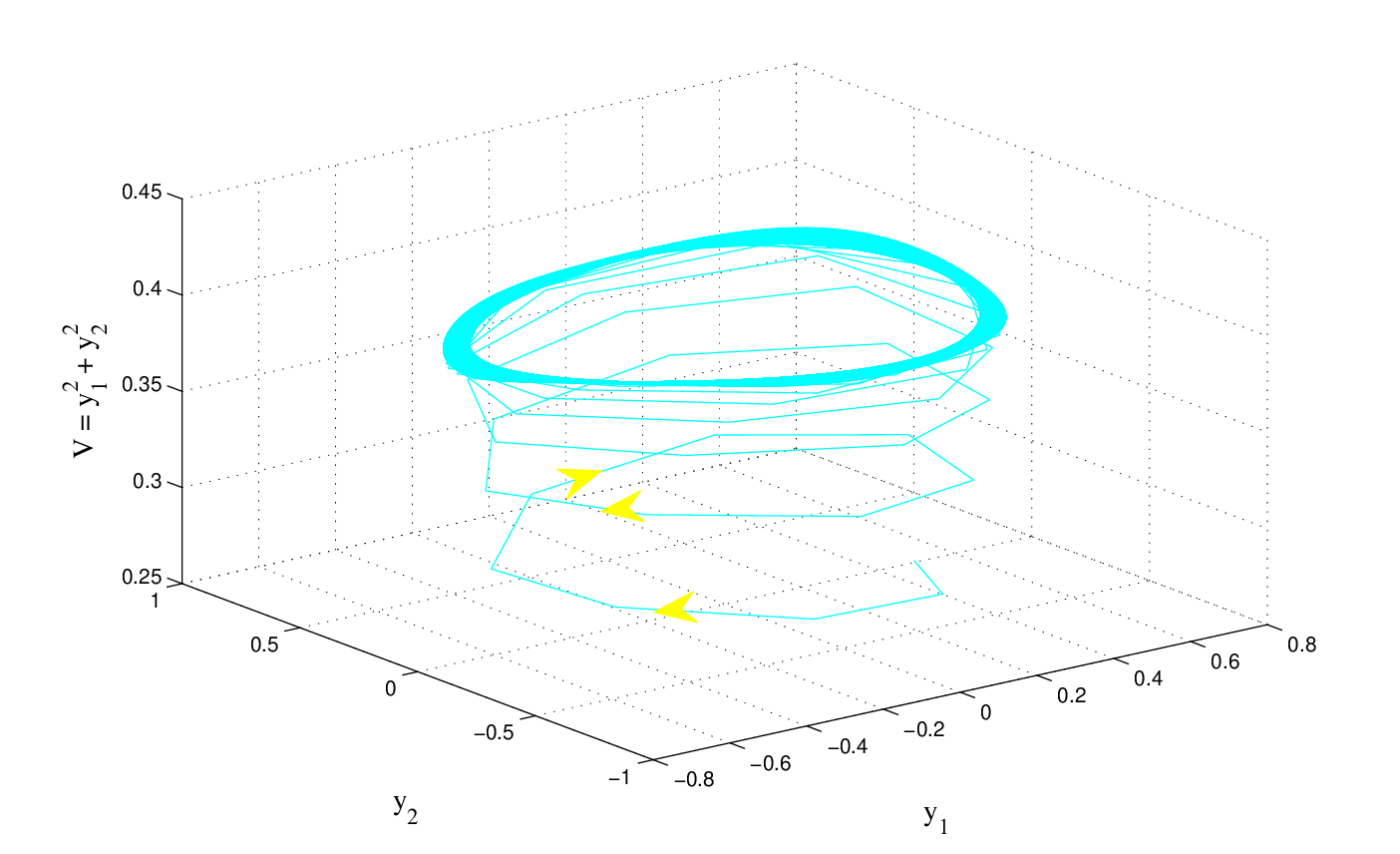}
\label{Figure:3d}
}\hfill
\subfloat[$\Delta V(y^k) = V(y^{k + 1}) - V(y^k)$.]{%
\includegraphics[height=9cm,width=12cm]{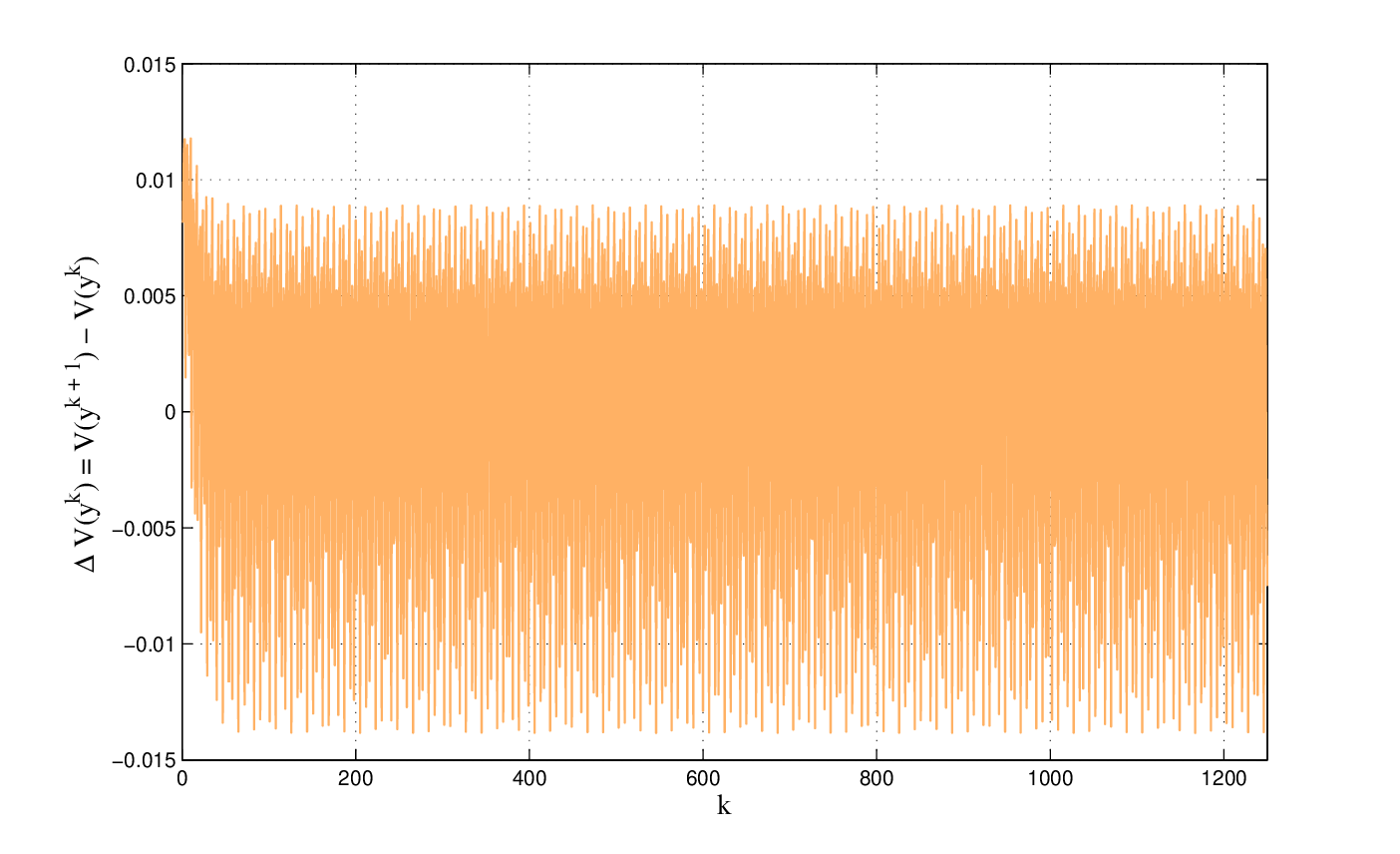}
\label{Figure:3e}
}
\caption{The discrete Lyapunov function and its variation corresponding to the RK2 method using $\Delta t = 0.8$.}
\label{Fig:3.1}
\end{figure}
\begin{figure}[H]
\subfloat[$\Delta t = 1.0$.]{%
\includegraphics[height=9cm,width=6cm]{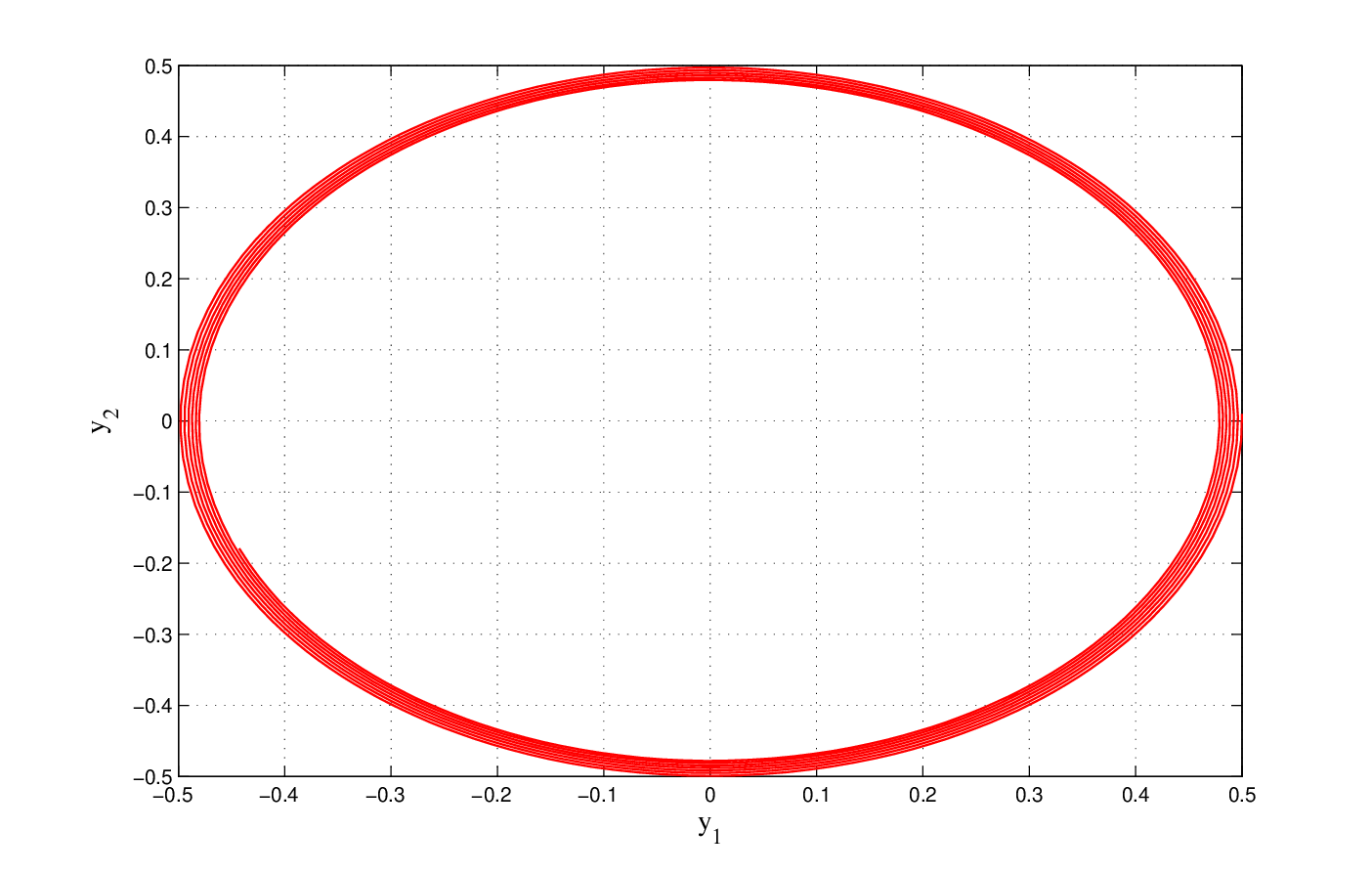}
\label{Figure:4a}
}\hfill
\subfloat[$\Delta t = 0.8$.]{%
\includegraphics[height=9cm,width=6cm]{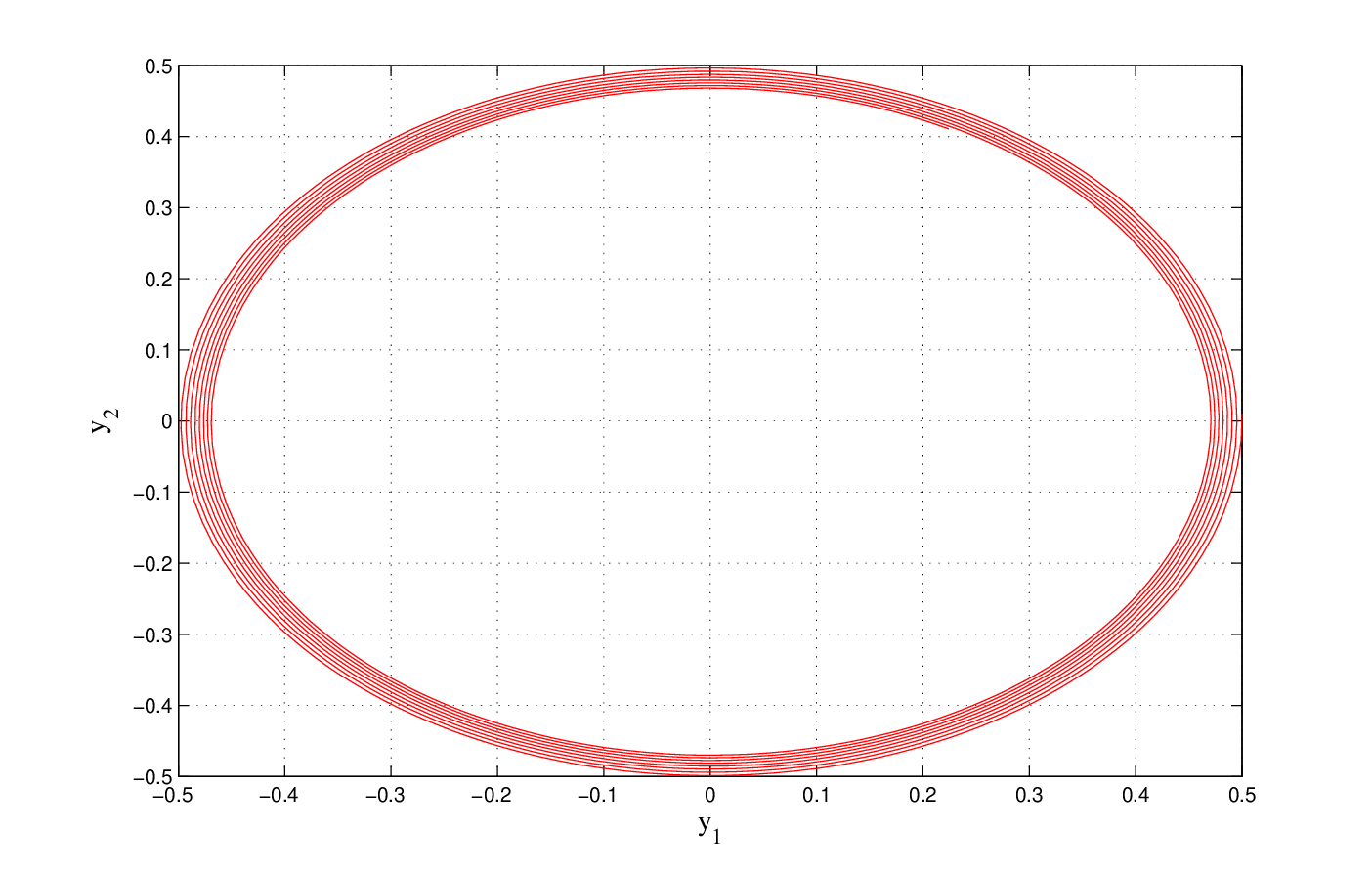}
\label{Figure:4b}
}\hfill
\subfloat[$\Delta t = 0.001$.]{%
\includegraphics[height=9cm,width=12cm]{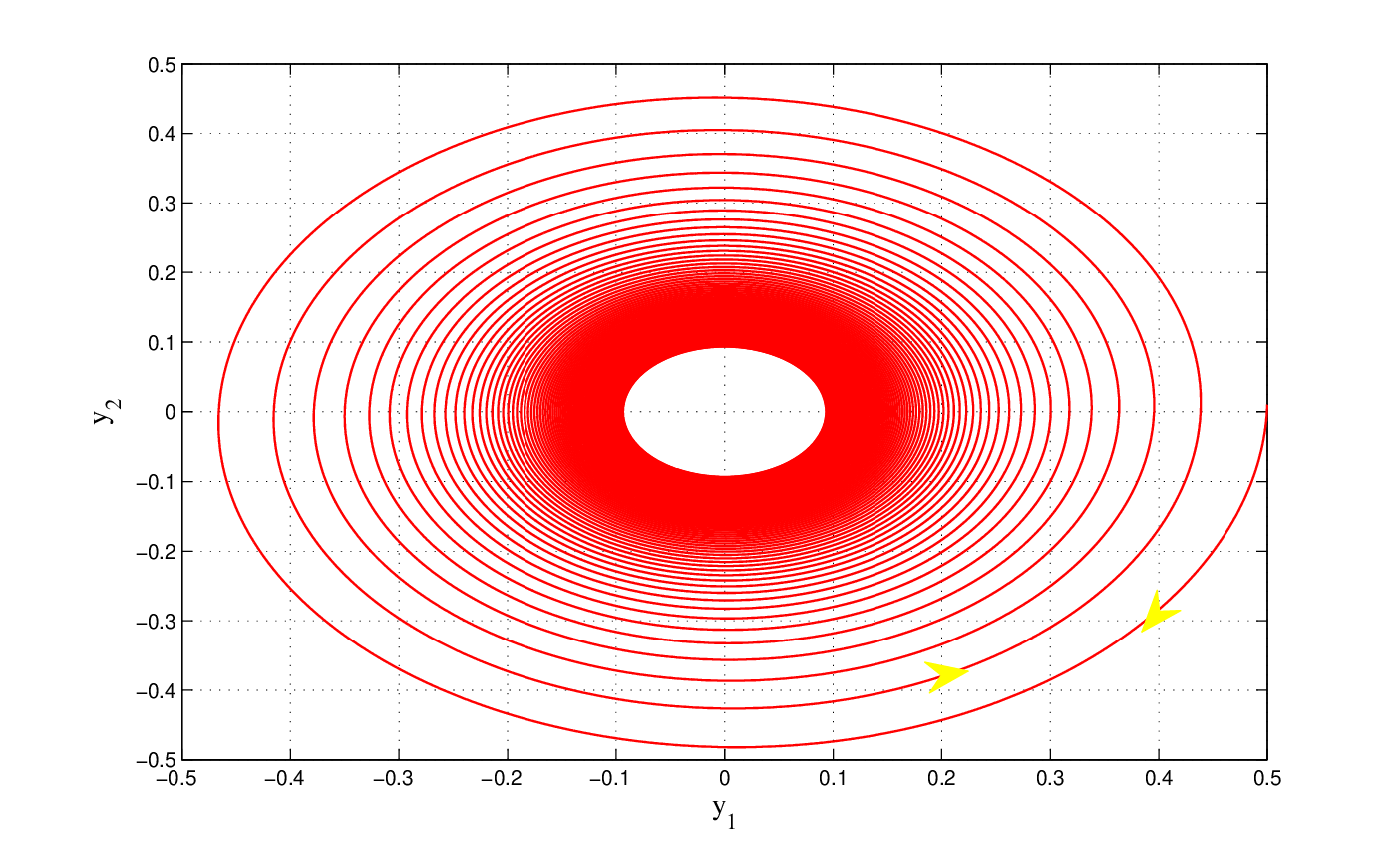}
\label{Figure:4c}
}
\caption{The approximations generated by employing the NSFD method.}
\label{Fig:4}
\end{figure}
\begin{figure}[H]
\subfloat[$V(y^k) = (y_1^k)^2 + (y_2^k)^2$.]{%
\includegraphics[height=9cm,width=12cm]{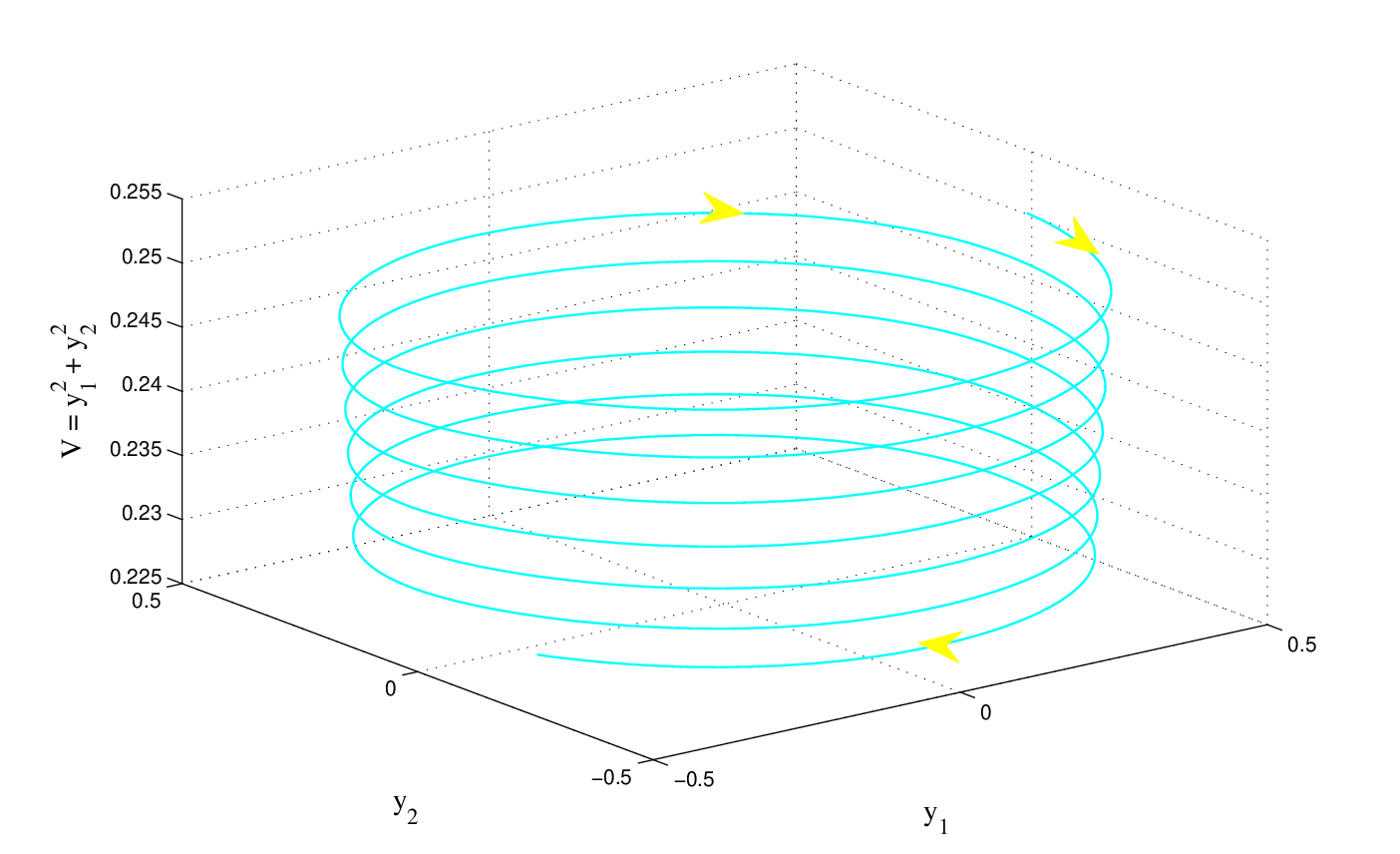}
\label{Figure:5a}
}\hfill
\subfloat[$\Delta V(y^k) = V(y^{k + 1}) - V(y^k)$.]{%
\includegraphics[height=9cm,width=12cm]{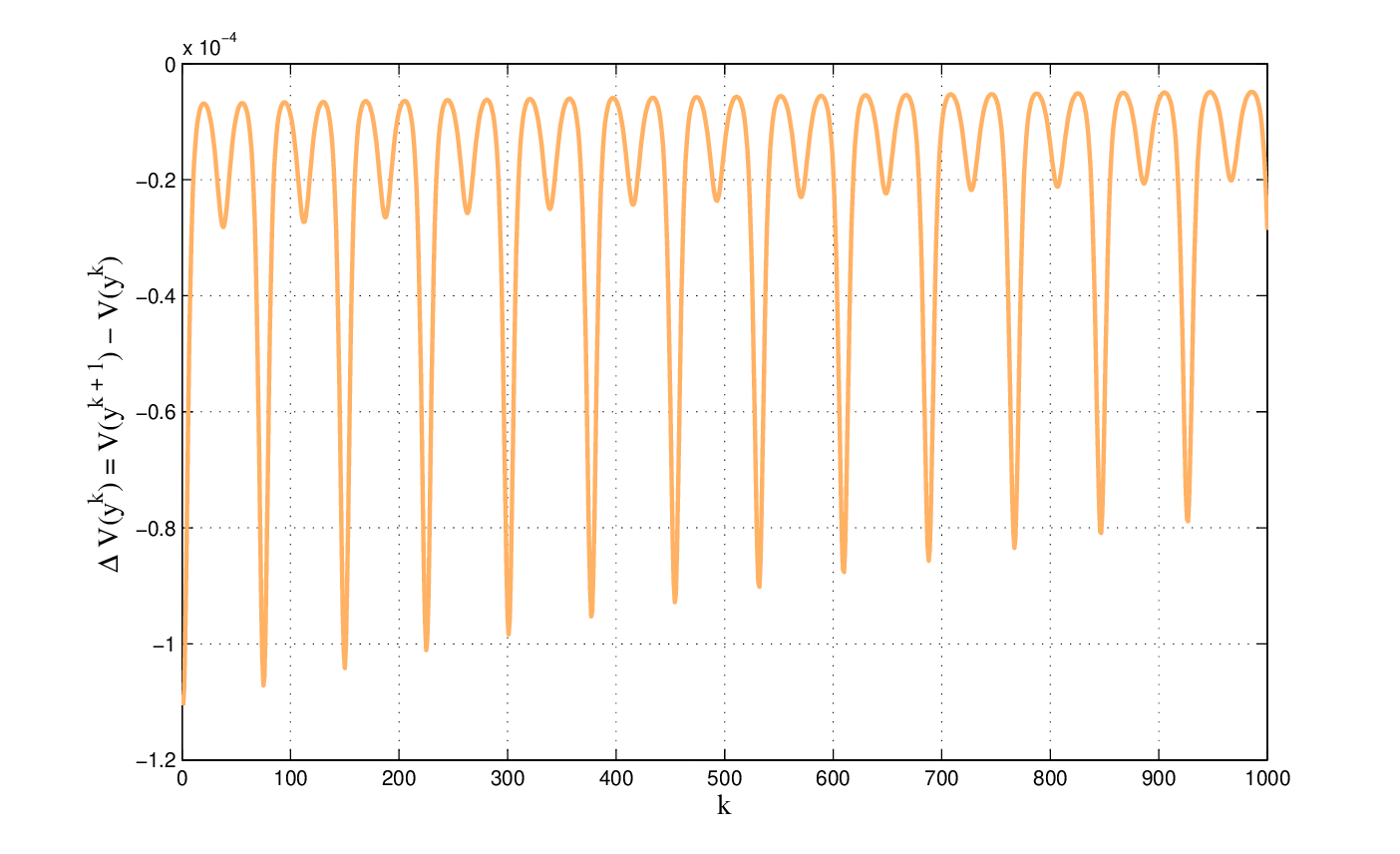}
\label{Figure:5b}
}
\caption{The discrete Lyapunov function and its variation corresponding to the NSFD method $\Delta t = 1.0$.}
\label{Fig:5}
\end{figure}
\begin{figure}[H]
\subfloat[$V(y^k) = (y_1^k)^2 + (y_2^k)^2$.]{%
\includegraphics[height=9cm,width=12cm]{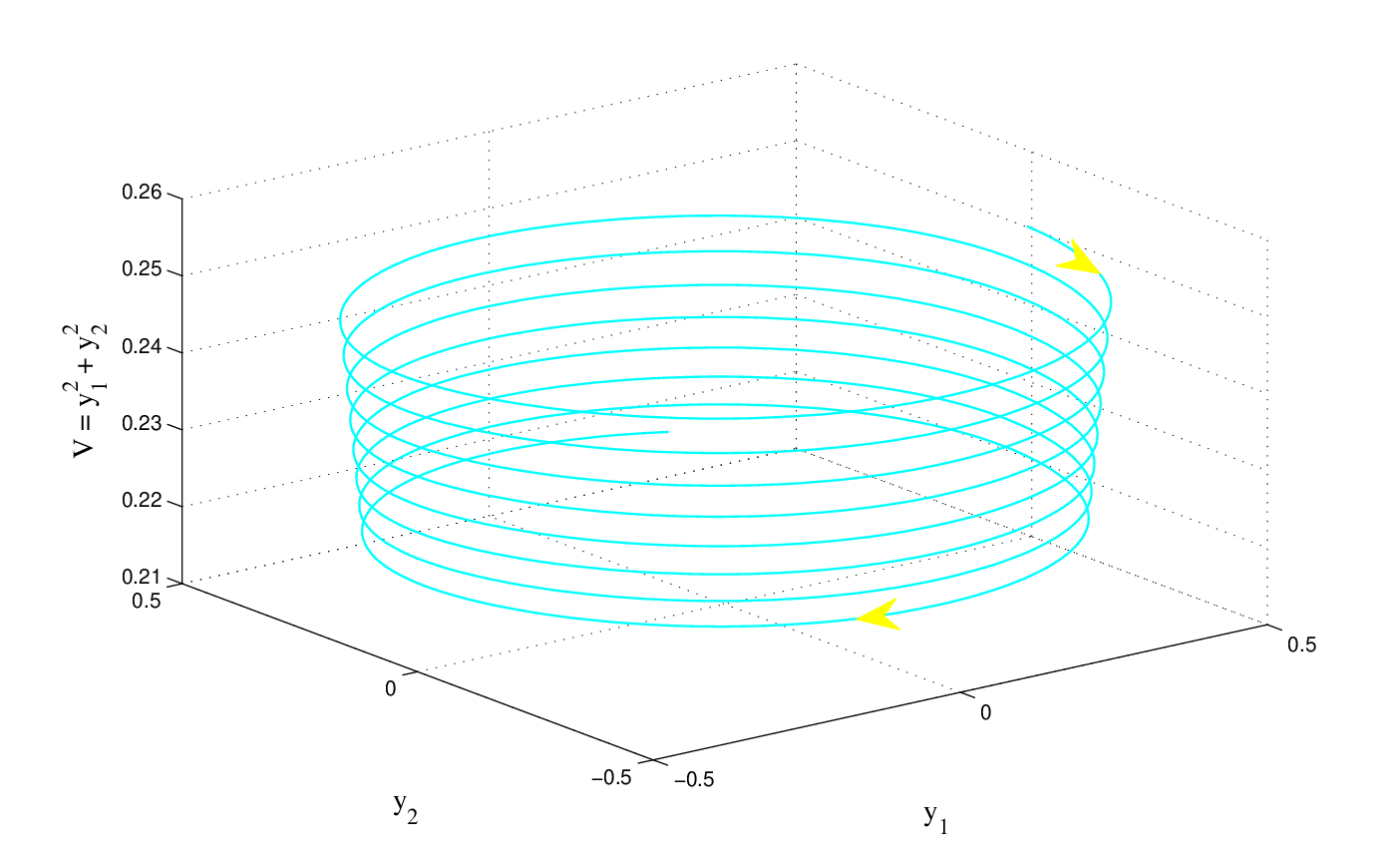}
\label{Figure:6a}
}\hfill
\subfloat[$\Delta V(y^k) = V(y^{k + 1}) - V(y^k)$.]{%
\includegraphics[height=9cm,width=12cm]{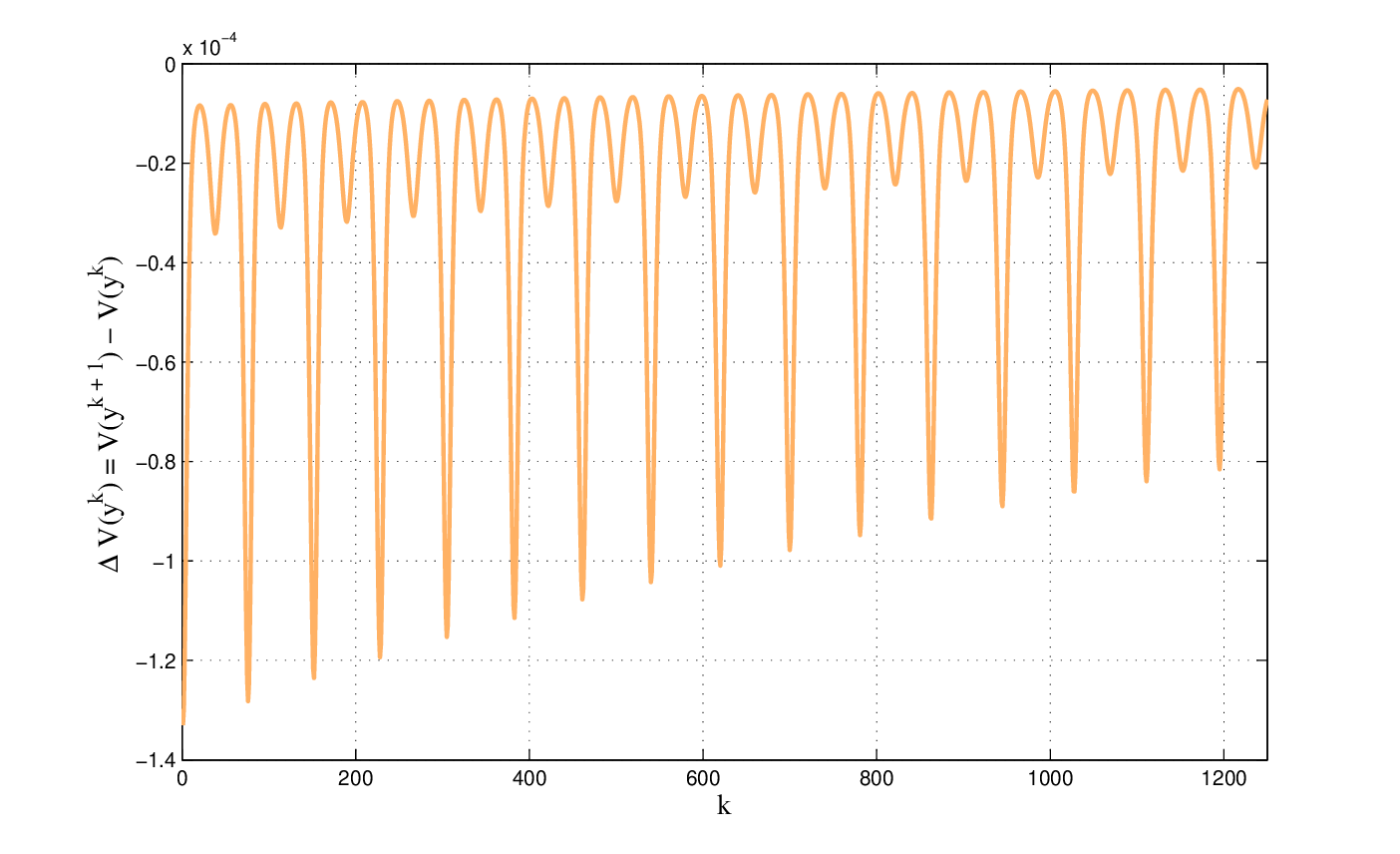}
\label{Figure:6b}
}
\caption{The discrete Lyapunov function and its variation corresponding to the NSFD method $\Delta t = 0.8$.}
\label{Fig:6}
\end{figure}
\begin{figure}[H]
\subfloat[$V(y^k) = (y_1^k)^2 + (y_2^k)^2$.]{%
\includegraphics[height=9cm,width=12cm]{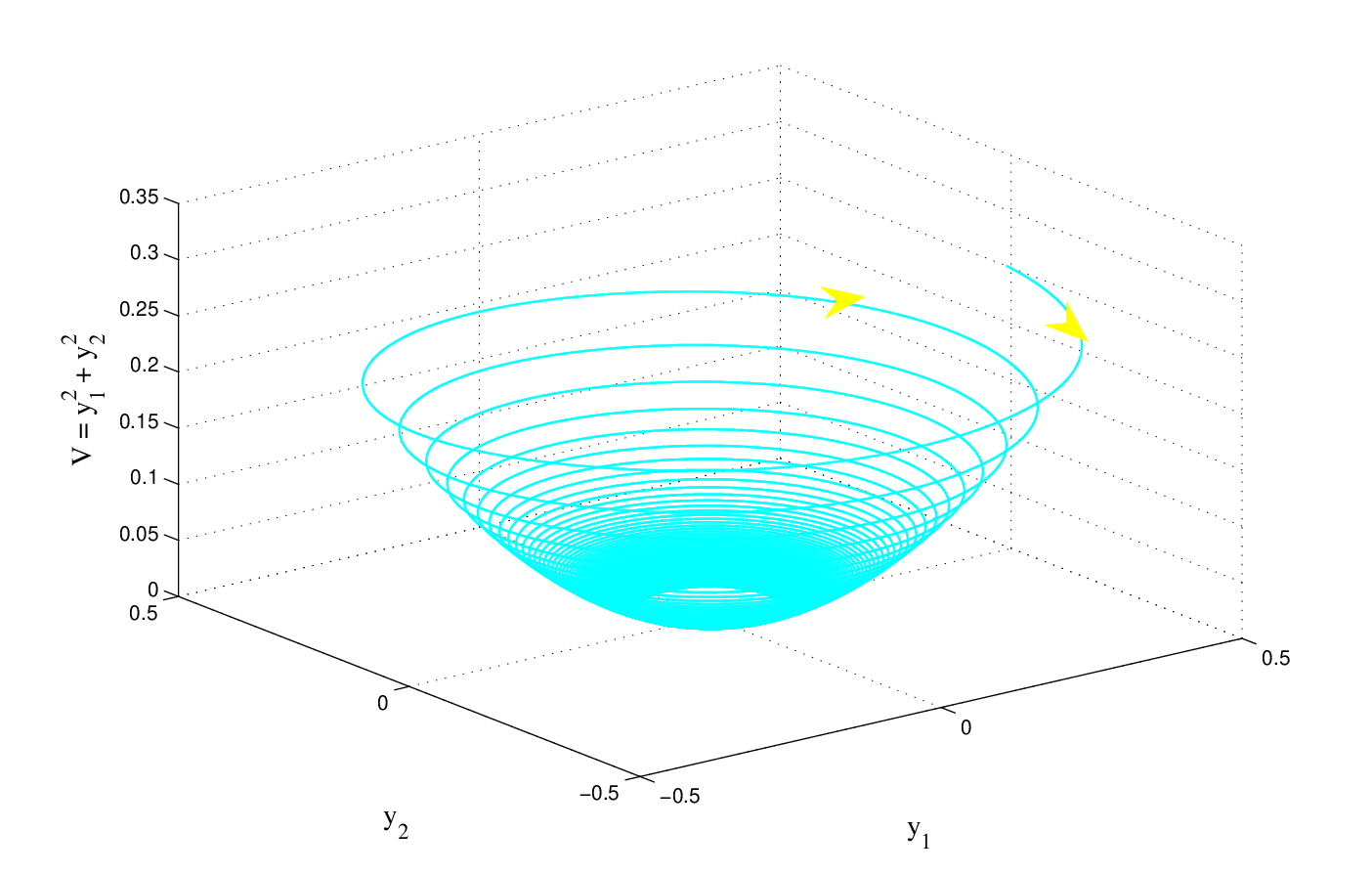}
\label{Figure:7a}
}\hfill
\subfloat[$\Delta V(y^k) = V(y^{k + 1}) - V(y^k)$.]{%
\includegraphics[height=9cm,width=12cm]{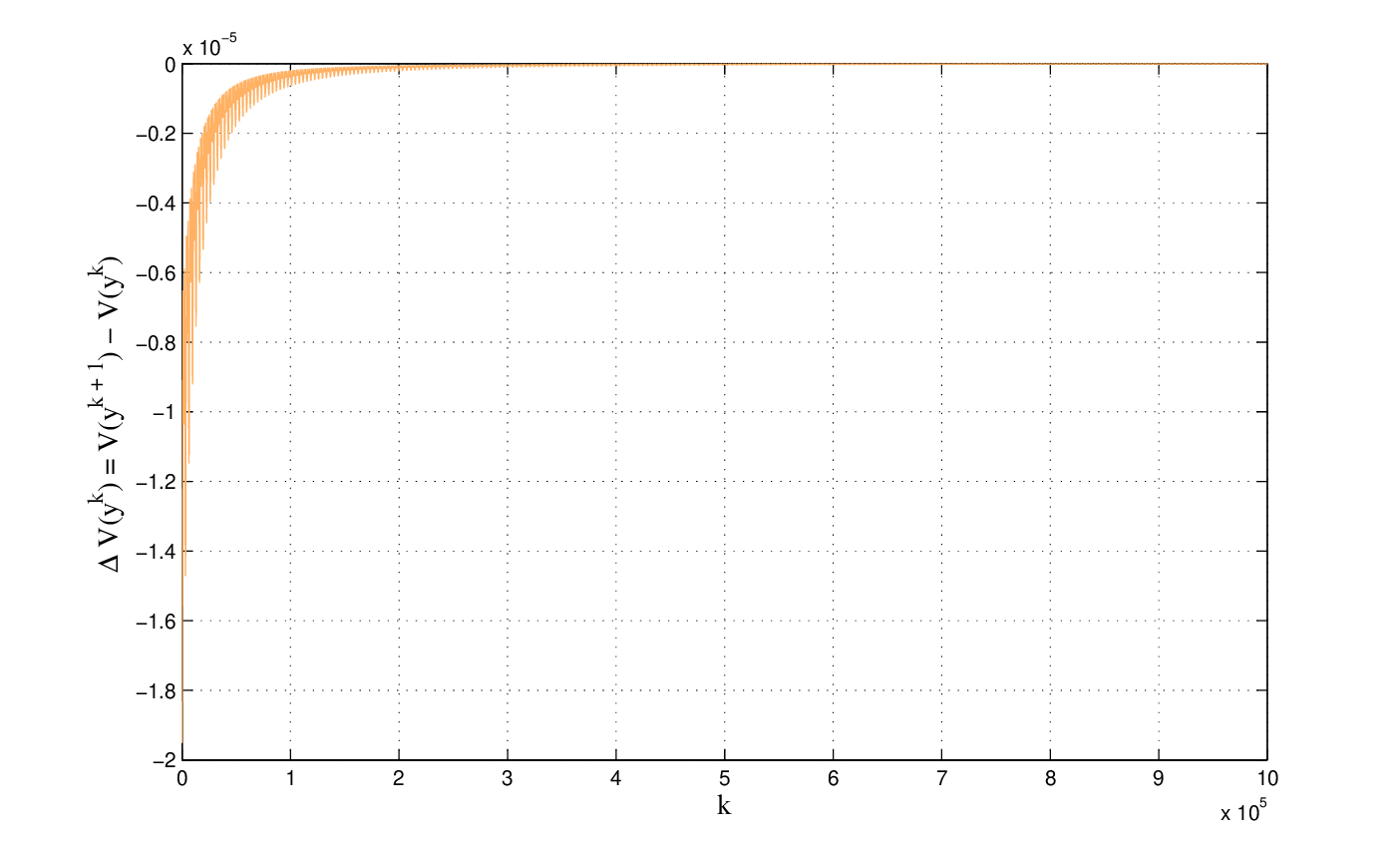}
\label{Figure:7b}
}
\caption{The discrete Lyapunov function and its variation corresponding to the NSFD method $\Delta t = 0.001$.}
\label{Fig:7}
\end{figure}

\begin{figure}[H]
\subfloat[$\Delta t = 1.0$.]{%
\includegraphics[height=9cm,width=12cm]{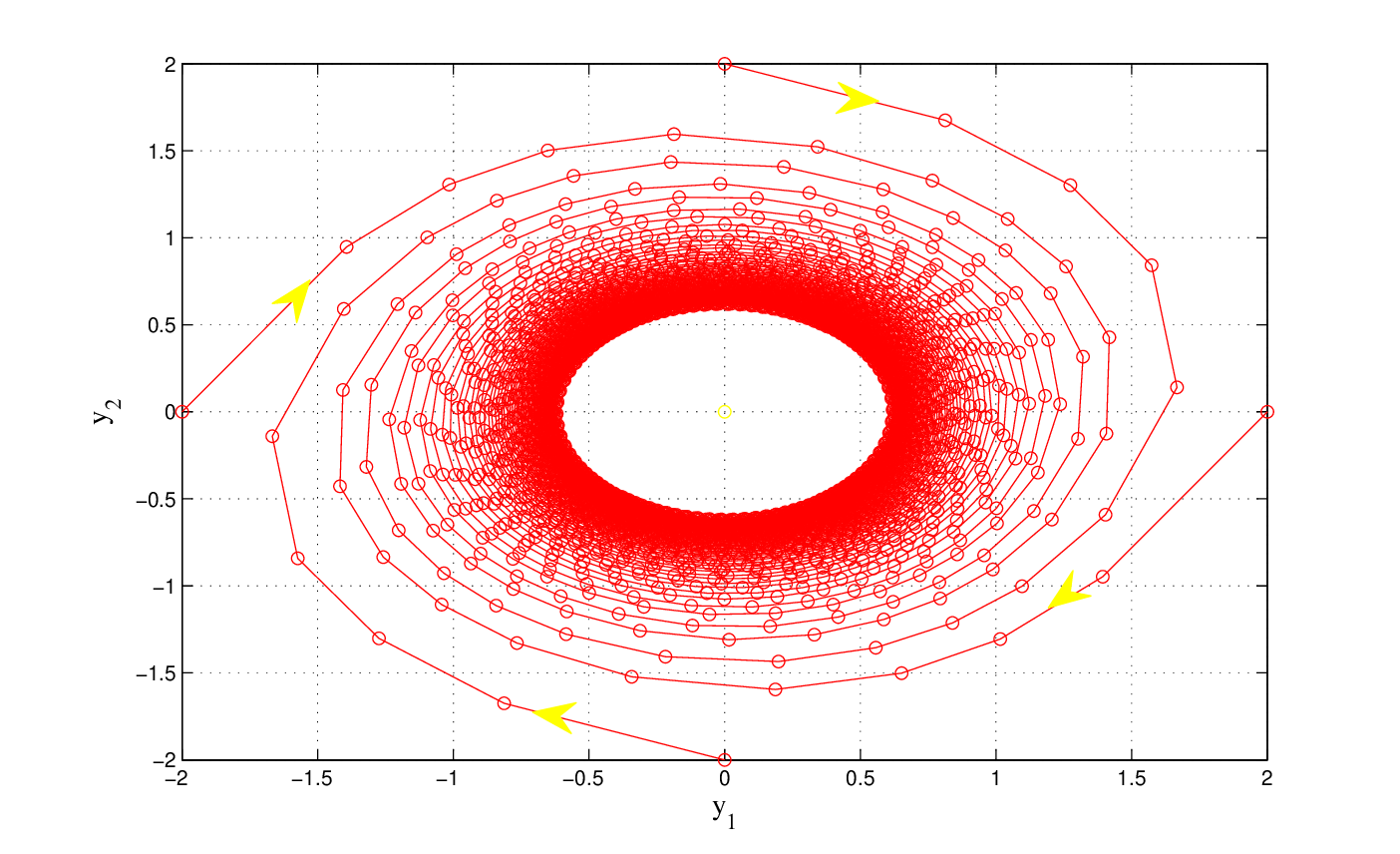}
\label{Figure:8}
}\hfill
\subfloat[$\Delta t = 0.001$.]{%
\includegraphics[height=9cm,width=12cm]{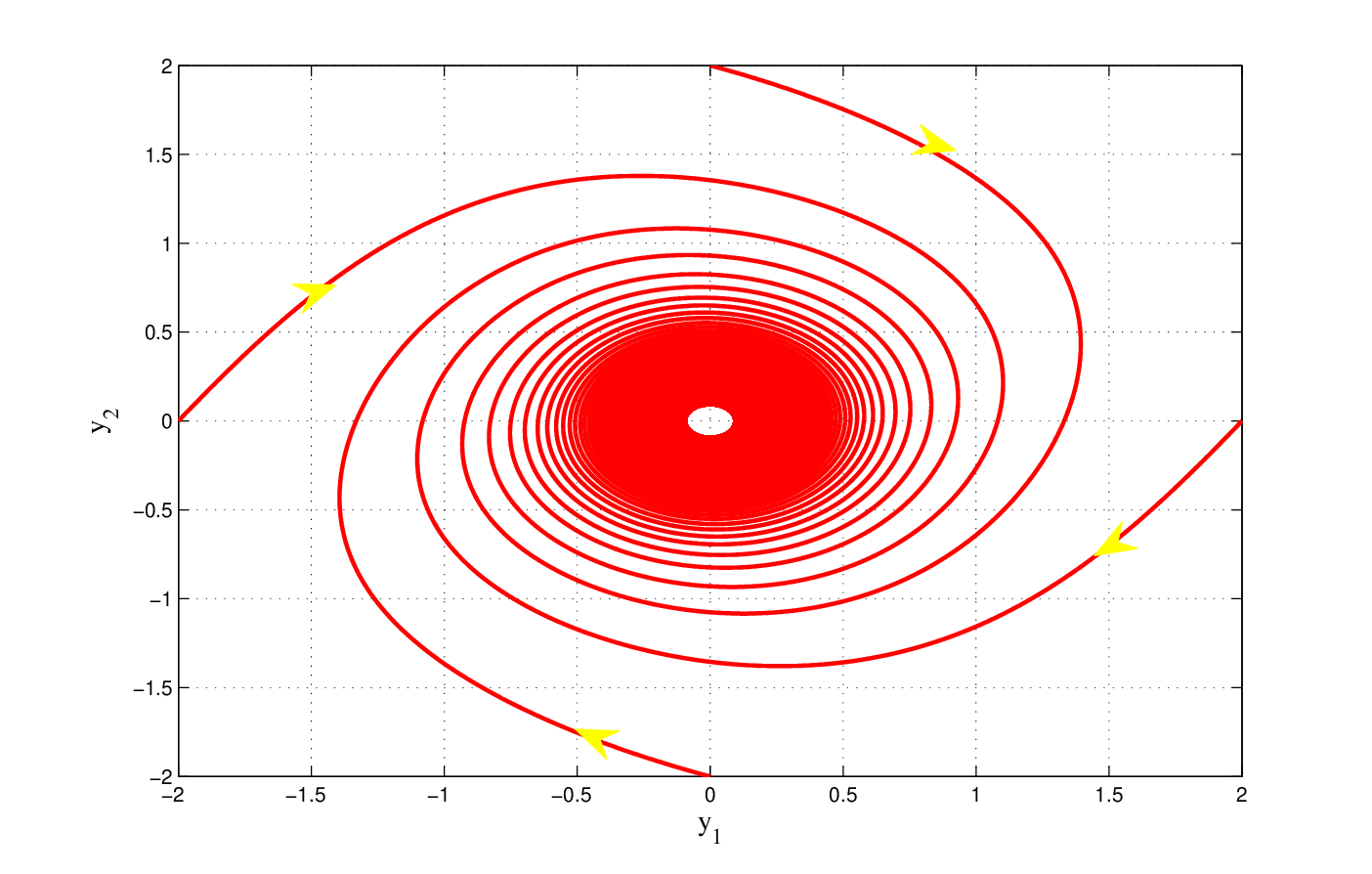}
\label{Figure:9}
}
\caption{The phase planes corresponding to some initial data generated by the NSFD method.}
\label{Fig:89}
\end{figure}
\section{On the positivity of NSFD methods}\label{sec4}
In this section, we investigate the positivity of the NSFD methods constructed in Section \ref{sec2}. Assume that \eqref{eq:1} admits the set
\begin{equation}\label{eq:17}
\mathbb{R}^n_+ = \big\{(y_1, y_2, \ldots, y_n) \in \mathbb{R}^n|y_1, y_2, \ldots, y_n \geq 0\big\}
\end{equation}
as a positively invariant set, that is $y(t) \geq 0$ whenever $y(0) \geq 0$. Here, the notation "$\geq$" are meant entry-wise for vectors. A necessary and sufficient condition for the set  $\mathbb{R}^n$
to be positively invariant for \eqref{eq:1} is (see \cite{Horvath, Smith})
\begin{equation}\label{eq:18}
f_i(y_1, y_2, \ldots, y_{i - 1}, 0, y_{i + 1}, \ldots, y_n) \geq 0, \quad \forall (y_1, y_2, \ldots, y_{i - 1}, y_{i + 1}, \ldots, y_n) \in \mathbb{R}_+^{n - 1}.
\end{equation}
We now give conditions for the NSFD model \eqref{eq:9} such that $y^0 = y(0) \geq 0$ implies $y^k \geq 0$  ($k > 0$) for all $\Delta t > 0$. Note that \eqref{eq:18} implies that if $f_i(y) < 0$, then $y_i^k > 0$.
\begin{theorem}\label{MainTheorem2}
Let $\tau_i: \mathbb{R}_+^n \to \mathbb{R}$ ($i = 1, 2, \ldots, n$) be functions satisfying
\begin{equation*}\label{eq:19}
\tau_i(y) \geq \tau_i^*(y) :=
\begin{cases}
&0 \quad \mbox{if} \quad f_i(y) \quad \geq 0,\\
&\\
&-\dfrac{f_i(y)}{y_i} \quad \mbox{if} \quad f_i(y) < 0
\end{cases}
\end{equation*}
and $\tau_P(y)$ be a function with the property that
\begin{equation*}\label{eq:19a}
\tau_P(y) \geq \max\big\{\tau_1^*(y), \tau_2^*(y), \ldots, \tau_n^*(y)\big\}.
\end{equation*}
Assume that $\tau(y): \mathbb{R}_+^n \to \mathbb{R}$ is a function satisfying
\begin{equation}\label{eq:19new}
\tau(y) \geq \tau_P(y).
\end{equation}
Then, the NSFD model \eqref{eq:9} admits the set $\mathbb{R}_+^n$ defined in \eqref{eq:17} as a positively invariant set for all the values of the step size $\Delta t$, that is $y^k \geq 0$ for $k \geq 1$ whenever $y(0) \geq 0$.
\end{theorem} 
\begin{proof}
This theorem is proved by mathematical induction. Assume that $y^k \geq 0$, we need to show that $y^{k + 1} \geq 0$. Indeed, we deduce from \eqref{eq:10} that $y_i^{k + 1} \geq 0$ if
\begin{equation}\label{eq:20}
f_i(y^k) + y_i^k\tau(y^k) \geq 0.
\end{equation}
It is clear that \eqref{eq:20} is automatically satisfied if $f_i(y^k) \geq 0$. If $f_i(y^k) < 0$, then it follows from  \eqref{eq:18} that $y_i^k > 0$. Hence, \eqref{eq:20} is satisfied whenever
\begin{equation*}
\tau(y^k) \geq -\dfrac{f_i(y^k)}{y_i^k} = \tau_i^*(y^k). 
\end{equation*}
Therefore, \eqref{eq:19new} implies the preservation of positivity of the NSFD model \eqref{eq:9}. The proof is completed.
\end{proof}
\begin{remark}\label{remark2}
Combining the results constructed in Sections \ref{sec2} and \ref{sec4}, we obtain the NSFD methods that preserve the quadratic Lyapunov functions and the positivity of solutions of the continuous-time dynamical systems model \eqref{eq:1}. More clearly, the NSFD model \eqref{eq:9} preserves the quadratic Lyapunov functions and the positivity of solutions of \eqref{eq:1} if the weight $\tau(y)$ is a function taking enough large values.
\end{remark}
\section{Concluding remarks and open problems}\label{sec5}
As the main conclusion of this work, we have introduced a simple approach to the construction of NSFD methods preserving the general quadratic Lyapunov functions. The proposed NSFD methods are derived from a novel approximation for the zero vector function. We have showed that  the constructed NSFD methods have the ability to preserve any given quadratic Lyapunov function regardless of the values of the step size, i.e. they admit $V$ as a discrete Lyapunov function As an important consequence, they are dynamically consistent w.r.t the GAS of the continuous-time dynamical systems. On the other hand, we study the positivity of the proposed NSFD methods. It is proved that they can also preserve the positivity of the solutions of the continuous-time dynamical systems. Finally, the theoretical findings are supported by a series of illustrative numerical experiments, in which advantages of the NSFD methods are shown.\par
Before ending this section, we discuss some open problems that studies in NSFD methods can take.\par
%
\begin{enumerate}
\item  NSFD methods preserving general Lyapunov functions: One of the limitations of the proposed NSFD methods is that they only preserve general quadratic Lyapunov functions, whereas many well-known Lyapunov functions have been proposed (see, for example, \cite{Cangiotti, Duarte-Mermoud, Korobeinikov, Korobeinikov1, Korobeinikov2, Korobeinikov3, ORegan, Vargas-De-Leon, Vargas-De-Leon1, Vargas-De-Leon2}). For example, general Voltera-type Lyapunov functions of the form \cite{Vargas-De-Leon}
\begin{equation*}
V_{VL}(y) := \sum_{i = 1}^n\alpha_i\bigg(y_i - y_i^* - y_i^*\ln\dfrac{y_i}{y_i^*}\bigg), \quad \alpha_i > 0,
\end{equation*}
which are appropriate for dynamical systems possessing positive solutions. Another example is the combination of the quadratic and Voltera-type Lyapunov functions \cite{Vargas-De-Leon}.
\item High-order NSFD methods preserving Lyapunov functions: The convergence of order $1$ of the proposed NSFD methods can be considered as another limitation of them.
\item NSFD methods preserving Lyapunov functions for fractional-order dynamical systems: In recent years, the Lyapunov stability theory for dynamical systems described by fractional-order differential equations has been strongly developed (see, e.g., \cite{Aguila-Camacho, Duarte-Mermoud, Li}). However, NSFD methods preserving Lyapunov functions for fractional-order dynamical systems have not been studied widely.
\end{enumerate}
\textbf{Ethical Approval:} Not applicable.\\
\textbf{Availability of supporting data:} The data supporting the findings of this study are available within the article [and/or] its supplementary materials.\\
\textbf{Conflicts of Interest:} We have no conflicts of interest to disclose.\\
\textbf{Funding information:} Not available.\\
\textbf{Acknowledgments:} Not applicable.


\end{document}